\documentclass[12pt]{amsart}

\usepackage[english]{babel}
\usepackage{tikz-cd}
\usepackage{amsrefs}
\usepackage{graphicx}
\graphicspath{ {images/} }
\usepackage[T1]{fontenc}
\usepackage{txfonts}
\usepackage{times}
\usepackage{amssymb,verbatim,amscd,amsfonts,amsmath,amsthm,enumerate}
\usepackage[all]{xy}
\usepackage{subfig}
\usepackage{tikz}
\usetikzlibrary{shapes,arrows,shadows}
\usetikzlibrary{decorations.markings}
\usepackage{color}
\usepackage[normalem]{ulem}
\usepackage{hyperref}
\usepackage{wrapfig}
\usetikzlibrary{arrows}
\usepackage{pb-diagram}

\newtheorem{theorem}{Theorem}[section]
\newtheorem{proposition}[theorem]{Proposition}
\newtheorem{corollary}[theorem]{Corollary}
\newtheorem{lemma}[theorem]{Lemma}

\newtheorem{remark}[theorem]{Remark}

\newtheorem{definition}[theorem]{Definition}

\begin{document}

%%%%%%%%%%%%%%%%%%%%%%%%%%%%%%%%%%%%%%%
%%%%%%%%%%%%%%%%%%%%%%%%%%
\renewcommand{\tablename}{Diagram}

%%%%%%%%%%%%%%%%%%%%%%%%%%%%%%%%%%%
%%%%%%%%%%%%%%%%%%%%%%%%%%%%%%%%%%%
%%%%%%Authors
%%%%%%%%%%%%%%%%%%%%%%%%%%%%%%%%%%
%%%%%%%%%%%%%%%%%%%%%%%%%%%%%%%%%%
\title{On the fiber product over infinite-genus Riemann surfaces}
\author{John A. Arredondo, Sa\'ul Quispe and Camilo Ram\'irez Maluendas}
\address{Facultad de Matem\'aticas e Ingenier\'ias, Fundaci\'on Universitaria Konrad Lorenz, Bogot\'a 110231, Colombia.}
\email{alexander.arredondo@konradlorenz.edu.co}
\address{Departamento de Matem\'atica y Estad\'{\i}stica, Universidad de La Frontera, Temuco 4780000, Chile.}
\email{saul.quispe@ufrontera.cl}
\address{Departamento de Matem\'aticas y Estad\'istica, Universidad Nacional de Colombia, Sede Manizales, Manizales 170004, Colombia.}
\email{camramirezma@unal.edu.co}
\thanks{The second author was partially supported by Project FONDECYT 1220261 and the third author was partially supported by Proyecto Hermes}
\maketitle

%%%%%%%%%%%%%%%%%%%%%%%%%%%%%%%%%%%%%%%%%%%
%%%%%%%%%%%%%%%%%%%%%%%%%%%%%%%%%%%%%%%%%%%
%%%%%%%%%%%Abstract
%%%%%%%%%%%%%%%%%%%%%%%%%%%%%%%%%%%%%%%%%%%
%%%%%%%%%%%%%%%%%%%%%%%%%%%%%%%%%%%%%%%%%%%

\begin{abstract}
Considering non-constant holomorphic maps $\beta_{i}:S_{i}\to S_{0}$, $i\in\{1,2\}$, between non-compact Riemann surfaces for which it is associated  its fiber product $S_{1}\times_{(\beta_{1},\beta_{2})}S_{2}$. With this setting, in this paper we relate the ends space of such fiber product to the ends space of its normal fiber product. Moreover, we provide conditions on the maps $\beta_{1}$ and $\beta_{2}$ to guarantee connectednes on the fiber product. From these conditions, we link the ends space of fiber product with the topological type of the Riemann surfaces $S_{1}$ and $S_{2}$. We also study the fiber product over infinite hyperelliptic curves and discuss its connectedness and ends space. 
\end{abstract}

%\textbf{Key words.} 

%\tableofcontents

%%%%%%%%%%%%%%%%%%%%%%%%%%%%%%%%%%%%%%%
%%%%%%%%%%%%%%%%%%%%%%%%%%%%%%%%%%%%%%%
%%%%%%%%%%%Introduction
%%%%%%%%%%%%%%%%%%%%%%%%%%%%%%%%%%%%%%%
%%%%%%%%%%%%%%%%%%%%%%%%%%%%%%%%%%%%%%

\section{Introduction}

In the schemes theory introduced by Grothendieck \cite{Gro60}, a scheme is a generalization of the notion of algebraic variety  allowing define varieties over any commutative ring.
Formally, a scheme is a topological space together with commutative rings for all of its open sets, which arises from gluing together spectra (spaces of prime ideals) of commutative rings along their open subsets. In other words, it is a ringed space which is locally a spectrum of a commutative ring \cite{Hart}. For the schemes $X_{1},X_{2}$ with morphisms $\beta_{1},\beta_{2}$ over a scheme $S$, it is define the  {\em fiber product} $X_{1}\times_{(\beta_{1},\beta_{2})} X_{2}$ as a new scheme with morphisms to $X_{1}$ and $X_{2}$ making a commutative diagram with any other scheme $Z$ with morphisms over $X_{1}$, $X_{2}$ and $X_{1}\times_{(\beta_{1},\beta_{2})} X_{2}$. As principal property, fiber products always exists in the category of schemes \cite{Iitaka}. Such schemes can be thought as set, variety, surface among others.

 The fiber product has been studied over compact Riemann surfaces \cite{Ruben}, \cite{Hidalgo-Reyes-Vega}. Whereas  along this paper, we will focus on the fiber product over non-compact Riemann surfaces, inasmuch as of Ker\'ekj\'art\'o's classification theorem of non-compact surfaces \cite{Ker},\cite{Ian}. The topological type of any surface $S$ (recall that we are assuming no boundary) is given by: (i) its genus $g\in \mathbb{N}\cup \{\infty\}$ and (ii) a couple of nested, compact, metrizable and totally disconnected spaces ${\rm Ends}_{\infty}(S)\subset {\rm Ends}(S)$. The set ${\rm Ends}(S)$ (respectively, ${\rm Ends}_{\infty}(S)$) is known as the ends space (respectively, the non-planar ends space) of $S$. Of all non-compact surfaces, as usually, we will focus on finite coverings of \emph{the Loch Ness monster (LNM)}, the unique, up to homeomorphism, infinite genus surface with exactly one end \cite{PSul}, see Figure \ref{Fig:LNM}. 

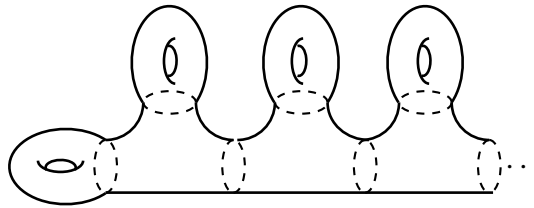
\begin{figure}[h!]
	\centering
	\begin{tikzpicture}[baseline=(current bounding box.north)]  
	\begin{scope}[scale=0.5]
	\clip (-5.8,-1.8) rectangle (9,3.9);
	%%%%torus1
	\draw[line width=1pt] (-3,0) arc (45:315:1.5 and 1);
	\draw [line width=1pt] (-4.8,-0.55) arc
	[
	start angle=180,
	end angle=360,
	x radius=6mm,
	y radius =3mm
	] ;
	\draw [line width=1pt] (-3.82,-0.8) arc
	[
	start angle=-20,
	end angle=200,
	x radius=4mm,
	y radius =2mm
	] ;
	%%%%%%%%%%%%%%%%%%%%%%%%%%
	%%%%%%%%%%%%
	
	%%%%%%%%%%%%%%%%%%%%%%%%%%%%%%%%%%%%%%%%
	\draw [dashed, line width=0.8pt] (-3,-1.4) arc
	[
	start angle=-90,
	end angle=90,
	x radius=3mm,
	y radius =7mm
	] ;
	\draw [dashed, line width=0.8pt] (-3,0) arc
	[
	start angle=90,
	end angle=270,
	x radius=3mm,
	y radius =7mm
	] ;
	\draw [line width=1pt] (-3,0) to[out=0,in=-90] (-2,1);
	\draw [dashed, line width=0.8pt] (-0.6,1) arc
	[
	start angle=0,
	end angle=180,
	x radius=7mm,
	y radius =3mm
	] ;
	\draw [dashed, line width=0.8pt] (-2,1) arc
	[
	start angle=180,
	end angle=360,
	x radius=7mm,
	y radius =3mm
	] ;
	%%%%%%%%%%%%%%%%%%%%%%%%%%%%%%%%%%%
	%%%%%%%% toro nuevo
	\draw [line width=1pt] (0.5,0) to[out=0,in=-90] (1.5,1);
	\draw [dashed, line width=0.8pt] (2.9,1) arc
	[
	start angle=0,
	end angle=180,
	x radius=7mm,
	y radius =3mm
	] ;
	\draw [dashed, line width=0.8pt] (1.5,1) arc
	[
	start angle=180,
	end angle=360,
	x radius=7mm,
	y radius =3mm
	] ;
	\draw[line width=1pt] (2.9,1) arc (-45:225:1 and 1.5);
	\draw [line width=1pt] (2.25,2.7) arc
	[
	start angle=90,
	end angle=270,
	x radius=3mm,
	y radius =6mm
	] ;
	\draw [line width=1pt] (2.1,2.5) arc
	[
	start angle=100,
	end angle=-100,
	x radius=2mm,
	y radius =4mm
	] ;
	\draw [line width=1pt] (2.9,1) to[out=-90,in=180] (4,0);
	%%%%%%%%%%%%%%%%%%%%%%%%%%
		%%%%%%%%%%
	\draw[line width=1pt] (-0.6,1) arc (-45:225:1 and 1.5);
	\draw [line width=1pt] (-1.15,2.7) arc
	[
	start angle=90,
	end angle=270,
	x radius=3mm,
	y radius =6mm
	] ;
	\draw [line width=1pt] (-1.35,2.5) arc
	[
	start angle=100,
	end angle=-100,
	x radius=2mm,
	y radius =4mm
	] ;
	%%%%%%%%%%%%%%%%%%%%%%%%%%%%%%%%%%%%%%%%
	\draw [line width=1pt] (-0.6,1) to[out=-90,in=180] (0.4,0);
	\draw [dashed, line width=0.8pt] (0.4,-1.4) arc
	[
	start angle=-90,
	end angle=90,
	x radius=3mm,
	y radius =7mm
	] ;
	\draw [dashed, line width=0.8pt] (0.4,0) arc
	[
	start angle=90,
	end angle=270,
	x radius=3mm,
	y radius =7mm
	] ;
	\draw [dashed, line width=0.8pt] (3.9,-1.4) arc
	[
	start angle=-90,
	end angle=90,
	x radius=3mm,
	y radius =7mm
	] ;
	\draw [dashed, line width=0.8pt] (3.9,0) arc
	[
	start angle=90,
	end angle=270,
	x radius=3mm,
	y radius =7mm
	] ;
	%%%%%%%%%%%%%%%%%%%%%%%%%%%%%%%
	\draw [line width=1pt](-3,-1.4) -- (7.3,-1.4);
	
	\node at (8.1,-0.7) {$\ldots$};
	%%%%%%%%%%%%%%%%%%%%%%%
	%%%%%%%%%%%%%%%%
	\draw [line width=1pt] (3.8,0) to[out=0,in=-90] (4.8,1);
	\draw[line width=1pt] (6.2,1) arc (-45:225:1 and 1.5);
	%%%%%%%%%%%%%%%%%%%%%%%%%%%%%%%%%%%
	\draw [line width=1pt] (5.6,2.7) arc
	[
	start angle=90,
	end angle=270,
	x radius=3mm,
	y radius =6mm
	] ;
	\draw [line width=1pt] (5.4,2.5) arc
	[
	start angle=100,
	end angle=-100,
	x radius=2mm,
	y radius =4mm
	] ;
	%%%%%%%%%%%%%%%%%%%%%%%%%%%%%%%%%%%%%%%%%%%
	\draw [dashed, line width=0.8pt] (6.2,1) arc
	[
	start angle=0,
	end angle=180,
	x radius=7mm,
	y radius =3mm
	] ;	
	\draw [dashed, line width=0.8pt] (6.2,1) arc
	[
	start angle=0,
	end angle=-180,
	x radius=7mm,
	y radius =3mm
	] ;	
	\draw [line width=1pt] (6.2,1) to[out=-90,in=180] (7.2,0);
	%%%%%%%%%%%%%%%%%%%%%%
	\draw [dashed, line width=0.8pt] (7.2,0) arc
	[
	start angle=90,
	end angle=270,
	x radius=3mm,
	y radius =7mm
	] ;
	\draw [dashed, line width=0.8pt] (7.2,0) arc
	[
	start angle=90,
	end angle=-90,
	x radius=3mm,
	y radius =7mm
	] ;	
	%%%%%%%%%%%%%%%%%%%%%%%%%%
	%%%%%%%%%%%%%%%%%%%%%%%%%%
	\end{scope}
	\end{tikzpicture}
	\caption{\emph{The Loch Ness monster.}}
	\label{Fig:LNM}
\end{figure}

 With the scenario introduced at this point, let us consider the fiber product $S_1\times_{(\beta_1,\beta_2)}S_2$ of the two pairs $(S_1,\beta_1)$ and $(S_2,\beta_2)$, where $S_0, S_1$ and $S_2$ are non-compact Riemann surfaces and $\beta_1:S_1\to S_0$ and $\beta_2:S_2\to S_0$ are non-constant holomorphic maps. In general $S_1\times_{(\beta_1,\beta_2)}S_2$ is a \textit{singular Riemann surface}, but these objects have associated its respective ends space. By deleting the \emph{locus of singular points} of the fiber product (this is a discrete set), we obtain a finite collection of analytically (finite) Riemann surfaces $\widetilde{R}_1,\ldots, \widetilde{R}_n$. Each $\widetilde{R}_j$ can be compactified (by adding its puntures) to obtain a unique, up to biholomorphism, Riemann surface called an \emph{irreducible component} of the fiber product. The union of all the irreducible components is called the \emph{normal fiber product} $  \widetilde{S_1\times_{(\beta_1, \beta_2)} S_2}$. In Theorem \ref{teo1} we describe the ends space of the fiber product from the point of view of the ends space of its respective irreducible components, when its locus of singular points is finite. More precisely, the result states that the ends space of the fiber product ${\rm Ends}(S_1\times_{(\beta_1, \beta_2)} S_2)$ is homeomorphic to the ends space of the normal fiber product ${\rm Ends}(  \widetilde{S_1\times_{(\beta_1, \beta_2)} S_2})$, if the locus of singular points is finite.

When $S_{0}$ is the Riemann sphere $\hat{\mathbb{C}}$, and the other two Riemann surfaces $S_{1}$ and $S_{2}$ are compact, W. Fulton and J. Hansen proved that the fiber product  $S_{1}\times_{(\beta_{1},\beta_{2})}S_{2}$ is connected \cite{Fulton-Hansen}*{Theorem p.160}. Nevertheless, if $S_{0}$ is a compact Riemann surface with genus, then the fiber product could be disconnected  \cite{Hidalgo-Reyes-Vega}*{Examples 2 and 3}. In Theorem \ref{theorem:ends_space_of_fiber_product}, we consider the non-compact Riemann surfaces $S_{1}$ and $S_{2}$, and the branched covering map $\beta_{i}:S_{i}\to\mathbb{C}$, such that $i\in\{1,2\}$; then we provide conditions on its branch points to guarantee connectedness on the fiber product $S_{1}\times_{(\beta_{1},\beta_{2})}S_{2}$. In Theorem \ref{theorem:ends_space_of_fiber_product2} we connect the end space of the fiber product ${\rm Ends}(S_{1}\times_{(\beta_{1},\beta_{2})}S_{2}$)  with the topological type of the non-compact Riemann surfaces involved, \emph{i.e.}, ${\rm Ends}(S_{1})$ and ${\rm Ends}(S_{2})$.

Different holomorphic structures on the LNM, which come form \emph{infinite hyperelliptic} and \emph{infinite superelliptic curves} were studied in  \cite{AGHQR}. It is well-known that if  $(w_{l})_{l\in\mathbb{N}}$ is a sequence of different complex numbers such that the its norm sequence $(\vert w_{l} \vert)_{l\in\mathbb{N}}$ diverges, then there is a Weierstrass's theorem ensuring the existence of an entire map $f:\mathbb{C}\to\mathbb{C}$ whose only zeros are the points of this sequence and, each one of them is simple. Therefore, for $n\geq 2$ the affine plane curve
$$
S(f)=\{(z_{1},z_{2}):z_{2}^{n}=f(z_{1})\},
$$
is called \emph{infinite superelliptic curve}. If $n=2$, the affine curve $S(f)$ is known as \emph{infinite hyperelliptic curve}. This infinite superelliptic curve is a Riemann surface topologically equivalent to the LNM. We consider an infinite superelliptic curve $S(f)$ and a suitable finite or infinite superelliptic curve $S(g)$; and the projection maps onto first coordinate $\beta_{1}:S(f)\to\mathbb{C}$ and $\beta_{2}:S(g)\to\mathbb{C}$, then we take the (singular Riemann surface) fiber product $\mathcal{S}(f,g):=S(f)\times_{(\beta_{1},\beta_{2})}S(g)$. Our  Theorem \ref{theorem:topology-of-fiber-product-superelliptic-curve} describes the connectedness of the fiber product and the ends space of $\mathcal{S}(f,g)$. In addition, in Theorem \ref{theorem:two-isomorphic-regular-riemann-surfaces} we determine  necessary and sufficient conditions guaranteeing that such two singular Riemann surfaces $\mathcal{S}(f,g)$ and $\mathcal{S}(f,h)$ are isomorphic. When the Riemann surfaces $S(f)$ and $S(g)$ are considered as infinite hyperelliptic curve, then the group $\mathbb{Z}_{2}\oplus\mathbb{Z}_{2}$ acts on the fiber product $\mathcal{S}(f,g)$ such that the quotients space $\mathcal{S}(f,g)/\mathbb{Z}_{2}\oplus\{Id\}$ and $\mathcal{S}/\{Id\}\oplus \mathbb{Z}_{2}$ are Riemann surfaces biholomorphic to $S(f)$ and $S(g)$ (see Lemma \ref{lemma:quotient}). Finally, in Proposition \ref{teo-cubrientes-dobles} we study the singular Riemann surface $\mathcal{S}(f,g)$ as covering of the curves $S(f):=\mathcal{S}(f,g)/\mathbb{Z}_{2}\oplus\{Id\}$ and $S(g):=\mathcal{S}/\{Id\}\oplus \mathbb{Z}_{2}$.

The paper is organized as follows. In Section \ref{sec:Preliminaries}, we recall the definition of ends and ends space of topological spaces, singular Riemann surfaces and fiber product. Sections \ref{sec:ends_spaces_fiber_product_riemann_surfaces} and \ref{sec:fiber_product_infinite_superelliptic_curve} are devoted to prove our main results.

%%%%%%%%%%%%%%%%%%%%%%%%%%%%%%%%%%%%%%%%%%%%%%
%%%%%%%%%%%%%%%%%%%%%%%%%%%%%%%%%%%%%%%%%%%%%%
%%%%%%%%%%%%%%%%%%%%%%%%%%%%%%%%%%%%%%%%%%%%%%%%%%%%%%%%%%%%%
\section{Preliminaries}\label{sec:Preliminaries}

\subsection{Ends of topological spaces.}\label{subsection:ends_space}  The idea of end was introduced by H. Freudenthal in \cite{Fre1}. Geometrically, an end of a suitable topological space is a point at infinity. We will briefly describe the concept of end for certain topological spaces and surfaces.

\begin{definition}[\cite{Fre1}*{1. Kapitel}]\label{definition:end}
Let $X$ be a locally compact, locally connected, connected, and Hausdorff space, and let $(U_n)_{n\in\mathbb{N}}$ be  an infinite nested sequence $U_{1}\supset U_{2}\supset\ldots$ of non-empty connected open subsets of $X$, such that the following hold:
\begin{itemize}
\item[\textbf{(1)}]  For each $n\in\mathbb{N}$, the boundary $\partial U_n$ of $U_{n}$ is compact, 

\item[\textbf{(2)}] The intersection $\cap_{n\in\mathbb{N}}\overline{U_{n}}=\emptyset
$, and

\item[\textbf{(3)}] For each compact  $K\subset X$ there is $m\in\mathbb{N}$ such that $K\cap U_m =\emptyset$.
\end{itemize}
Two nested sequences $(U_n)_{n\in\mathbb{N}}$ and $(U'_{n})_{n\in\mathbb{N}}$ are {\it equivalent} if for each $n\in\mathbb{N}$ there exist $j,k\in\mathbb{N}$ such that $U_{n}\supset U'_{j}$, and $U'_{n}\supset U_{k}$. The corresponding equivalence classes $[U_{n}]_{n\in\mathbb{N}}$ of these sequences are called the {\it ends} of $X$, and the set of all ends of $X$ is denote by ${\rm Ends}(X)$.
\end{definition}

\begin{remark}\label{remarK:different-classes}
 If the ends $[U_{n}]_{n\in\mathbb{N}}$ and $[V_{n}]_{n\in\mathbb{N}}$ of $X$ are different, then there are $l,m\in\mathbb{N}$ such that $U_{l}\cap V_{m}=\emptyset$. One can suppose without of generality that $U_{1}\cap V_{1}=\emptyset$.
\end{remark}

The {\it ends space} of $X$ is the topological space having the ends of $X$ as elements, and endowed with the following topology: for every non-empty connected open subset $U$ of
$X$ such that its boundary $\partial U$ is compact, we define
\begin{equation*}\label{eq:end_open}
U^{*}:=\{[U_{n}]_{n\in\mathbb{N}}\in{\rm Ends}(X)\hspace{1mm}|\hspace{1mm}U_{j}\subset U\hspace{1mm}\text{for some }j\in\mathbb{N}\}.
\end{equation*}
Then we take the set of all such $U^{\ast}$, with $U$ open with compact boundary in $X$, as a basis for the topology of ${\rm Ends}(X)$. It is easy check that if $U\subset V$ are open subsets of $X$ with compact boundary, then $U^{\ast}\subset V^{\ast}$.

\begin{theorem}[\cite{Ray}*{Theorem 1.5}]
The space ${\rm Ends}(X)$ is Hausdorff, totally disconnected,
and compact. In other words, the end space ${\rm Ends}(X)$ is a closed subset of the Cantor set.
\end{theorem}

\begin{lemma}[\cite{SPE}*{\S 5.1., p. 320}]\label{lemma:spec} 
	The space $X$ has exactly $n\in\mathbb{N}$ ends if and only if for all compact subset $K \subset X$ there is a compact subset $K^{'}\subset X$ such that $K\subset K^{'}$ and $X \setminus  K^{'}$ are $n$ component connected.
\end{lemma}

%%%%%%%%%%%%%%%%%%%%%%%%%%%%%%%%%%%%%%%%%%%%%%%
%%%%%%%%%%%%%%%%%%%%%%%%%%%%%%%%%%%%%%%%%%%%%%%%
\subsubsection{Ends of surfaces}\label{subsection:ends_surfaces} Topological orientable surfaces are classified, up to homeomorphisms, by their genus $g(S)\in \mathbb{N}_{0}\cup\{\infty\}$, the ends space ${\rm Ends}(S)$  and, the subspace ${\rm Ends}_{\infty}(S)\subseteq {\rm Ends}(S)$ of all \emph{non-planar ends space} of $S$ (or, ends accumulated by genus). In addition, any pair of nested closed subsets of the Cantor set can be realized as the space of ends of a connected, orientable topological surface. For more details, we refer the reader to \cite{Ian}.

\begin{theorem}[Classification of topological surfaces, \cite{Ker}*{\S 7}, \cite{Ian}*{Theorem 1}]\label{Thm:ClassificationOfSurfaces}
	Two orientable surfaces $S_1$ and $S_2$  having the same genus are topological equivalent if and only if there exists a homeomorphism $f: {\rm Ends}(S_1)\to {\rm Ends}(S_2)$ such that $f( {\rm Ends}_{\infty}(S_1))= {\rm Ends}_{\infty}(S_2)$.
\end{theorem}

%%%%%%%%%%%%%%%%%%%%%%%%%%%%%%%%%%%%%%%%%%%%%%%%%%%%%%%%%%%%%%%%%%%%%%%%%%%%%%%%%%%%%%
%%%%%%%%%%%%%%%%%%%%%%%%%%%%%%%%%%%%%%%%%%%%%%%%%%%%%%%%%%%%%%%%%%%%%%%%%%%%%%%%%%%%

\subsection{Singular Riemann surfaces}\label{Subsection:singular_riemann_surface} In this subsection, we shall explore some elements of the theory of the singular Riemann surfaces as the locus of singular points, irreducible components, isomorphism  and the group of autormophism of a singular Riemann surface. The Poincar\'e disc will be denoted by $\Delta$. 

\begin{definition}[\cite{Hidalgo-Reyes-Vega}*{Subsection 2.2}]
A \textbf{singular Riemann surface} is an one-dimensional complex analytic\footnote{It means a Riemann surface that allows singularities.} surface $S$, such that for each point $p$ of $S$ there exists a neighborhood holomorphically equivalent to a subspace of the form
\[
V_{n,m}=\{(z,w)\in \Delta \times \Delta: z^n=w^m\}\subset \Delta\times\Delta,
\]
for some integers $n,m\geq 1$. 
\end{definition}
If $n=1$ or $m=1$, then $V_{n,m}$ is holomorphically equivalent to $\Delta$. Now, if $n,m\geq 2$ and $d\geq 1$ is the greatest common divisor of $n$ and $m$, then we write $n=d\widehat{n}$ and $m=d\widehat{m}$, where  $\widehat{n}, \widehat{m}\geq 1$ are relatively prime integers. Thus the space $V_{n,m}$ can be written as
$$V_{n, m}=\left\{(z, w)\in \Delta\times \Delta:\ \prod_{k=0}^{d-1}(z^{\widehat{n}}-\omega^kw^{\widehat{m}})=0\right\},$$ 
and it is homeomorphic to a collection of $d$ cones with common vertex at $(0,0)$ (see Figure \ref{fig:titulo}), where $\omega$ is a $d$-th primitive root of unity. In particular, if $d=1$, then the space $V_{n,m}$ is holomorphically equivalent to $\Delta$.

If $d\geq 2$, then the point $p\in S$ is called {\em singular}, and the \emph{locus of singular points of $S$}, denoted by
${\rm Sing}(S)$ is a discrete subset of $S$. It follows that each connected component $\widetilde{R}$ of $S\setminus {\rm Sing}(S)$ has structure of Riemann surface, and the points in ${\rm Sing}(S)$ define punctures on $\widetilde{R}$.  By adding these punctures, we obtain another Riemann surface $R$, containing $\widetilde{R}$, called an {\em irreducible component} of $S$. If $S$ has only one irreducible component then it is called {\em irreducible}; otherwise, it is called {\em reducible}.

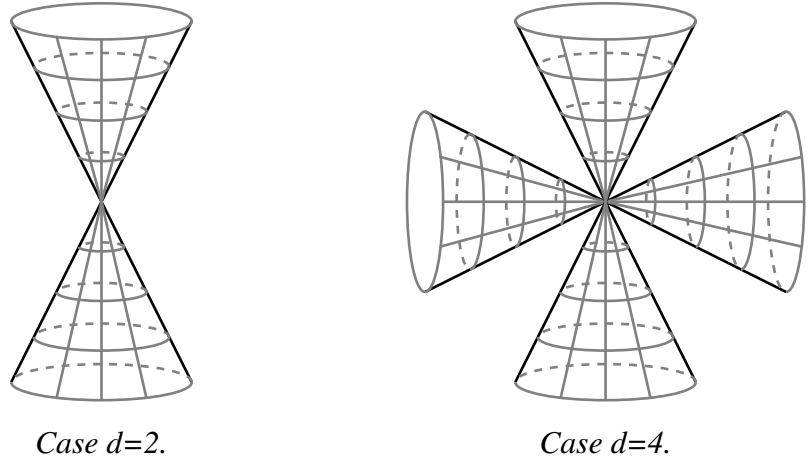
\begin{figure}[ht]
	\centering
	\begin{tabular}{ccc}	
    \begin{tikzpicture}[baseline=(current bounding box.north)]
		\begin{scope}[scale=0.6]
		\clip (-5,-5) rectangle (5,5);
		%%%%%%%%%%%%%%%%%%%%%%%%%%%%%%%%%%%%%%%%%%%%%%%
		%%
		\draw [line width=1pt] (0,0) -- (-2,4);
		\draw [line width=1pt] (0,0) -- (2,4);
		\draw [line width=1pt] (0,0) -- (-2,-4);
		\draw [line width=1pt] (0,0) -- (2,-4);
		%%%%%%%%%%%%%%%%%
		\draw [line width=1pt,black!50] (0,0) -- (0,3.6);
		\draw [line width=1pt,black!50] (0,0) -- (0,-4.4);
		\draw [line width=1pt,black!50] (0,0) -- (1,3.68);
		\draw [line width=1pt,black!50] (0,0) -- (-1,3.68);
		\draw [line width=1pt,black!50] (0,0) -- (1,-4.37);
		\draw [line width=1pt,black!50] (0,0) -- (-1,-4.37);
		%%%%%%%%%%%%
		\draw [line width=1pt,black!50] (0,4) ellipse (20mm and 4mm);
		\draw [line width=1pt,black!50] (1.5,3) arc
		[
		start angle=0,
		end angle=-180,
		x radius=15mm,
		y radius =3mm
		] ;
		\draw [dashed, line width=1pt,black!50] (1.5,3) arc
		[
		start angle=0,
		end angle=180,
		x radius=15mm,
		y radius =3mm
		] ;
		%%%%%%%%%%%
		\draw [line width=1pt,black!50] (1,2) arc
		[
		start angle=0,
		end angle=-180,
		x radius=10mm,
		y radius =2mm
		] ;
		\draw [dashed, line width=1pt,black!50] (1,2) arc
		[
		start angle=0,
		end angle=180,
		x radius=10mm,
		y radius =2mm
		] ;
		%%%%%%%%%%%%%%%%%%%%%%%%%%%%%%%%%%%%%%%%%%%%%
		\draw [line width=1pt,black!50] (0.5,1) arc
		[
		start angle=0,
		end angle=-180,
		x radius=5mm,
		y radius =1mm
		] ;
		\draw [dashed, line width=1pt,black!50] (0.5,1) arc
		[
		start angle=0,
		end angle=180,
		x radius=5mm,
		y radius =1mm
		] ;
		%%%%%bAJAS
		%\draw [line width=1pt] (0,-4) ellipse (20mm and 4mm);
		%%%%%   
		\draw [line width=1pt,black!50] (2,-4) arc
		[
		start angle=0,
		end angle=-180,
		x radius=20mm,
		y radius =4mm
		] ;
		\draw [dashed, line width=1pt,black!50] (2,-4) arc
		[
		start angle=0,
		end angle=180,
		x radius=20mm,
		y radius =4mm
		] ;
		%%%%%%%%%%%
		\draw [line width=1pt,black!50] (1.5,-3) arc
		[
		start angle=0,
		end angle=-180,
		x radius=15mm,
		y radius =3mm
		] ;
		\draw [dashed, line width=1pt,black!50] (1.5,-3) arc
		[
		start angle=0,
		end angle=180,
		x radius=15mm,
		y radius =3mm
		] ;
		%%%%%%%%%%%
		\draw [line width=1pt,black!50] (1,-2) arc
		[
		start angle=0,
		end angle=-180,
		x radius=10mm,
		y radius =2mm
		] ;
		\draw [dashed, line width=1pt,black!50] (1,-2) arc
		[
		start angle=0,
		end angle=180,
		x radius=10mm,
		y radius =2mm
		] ;
		%%%%%%%%%%%%%%%%%%%%%%%%%%%%%%%%%%%%%%%%%%%%%
		\draw [line width=1pt,black!50] (0.5,-1) arc
		[
		start angle=0,
		end angle=-180,
		x radius=5mm,
		y radius =1mm
		] ;
		\draw [dashed, line width=1pt,black!50] (0.5,-1) arc
		[
		start angle=0,
		end angle=180,
		x radius=5mm,
		y radius =1mm
		] ;
		%%%%% 
		
		%%%%%%%%%%
		%%%%%%%%%%%%%%%%%%%%%%%%%
		%\draw [line width=0.8pt, ->, >=latex, black!30] (-3.5,0) -- (3.5,0);
		%\draw [line width=0.8pt, ->, >=latex, black!30] (1.2,1.2) -- (-2.5,-2.5);
		%\draw [line width=0.8pt, ->, >=latex, black!30] (0,-3.5) -- (0,3.5);
		\end{scope}
		\end{tikzpicture}
  &&		\begin{tikzpicture}[baseline=(current bounding box.north)]
		\begin{scope}[scale=0.6]
		\clip (-5,-5) rectangle (5,5);
		%%%%%%%%%%%%%%%%%%%%%%%%%%%%%%%%%%%%%%%%%%%%%%%
		%%
		\draw [rotate=90,line width=1pt] (0,0) -- (-2,4);
		\draw [rotate=90,line width=1pt] (0,0) -- (2,4);
		\draw [rotate=90,line width=1pt] (0,0) -- (-2,-4);
		\draw [rotate=90,line width=1pt] (0,0) -- (2,-4);
		%%%%%%%%%%%%%%%%%
		\draw [rotate=90,line width=1pt,black!50] (0,0) -- (0,3.6);
		\draw [rotate=90,line width=1pt,black!50] (0,0) -- (0,-4.4);
		\draw [rotate=90,line width=1pt,black!50] (0,0) -- (1,3.68);
		\draw [rotate=90,line width=1pt,black!50] (0,0) -- (-1,3.68);
		\draw [rotate=90,line width=1pt,black!50] (0,0) -- (1,-4.37);
		\draw [rotate=90,line width=1pt,black!50] (0,0) -- (-1,-4.37);
		%%%%%%%%%%%%
		\draw [rotate=90,line width=1pt,black!50] (0,4) ellipse (20mm and 4mm);
		\draw [rotate=90,line width=1pt,black!50] (1.5,3) arc
		[
		start angle=0,
		end angle=-180,
		x radius=15mm,
		y radius =3mm
		] ;
		\draw [rotate=90,dashed, line width=1pt,black!50] (1.5,3) arc
		[
		start angle=0,
		end angle=180,
		x radius=15mm,
		y radius =3mm
		] ;
		%%%%%%%%%%%
		\draw [rotate=90,line width=1pt,black!50] (1,2) arc
		[
		start angle=0,
		end angle=-180,
		x radius=10mm,
		y radius =2mm
		] ;
		\draw [rotate=90,dashed, line width=1pt,black!50] (1,2) arc
		[
		start angle=0,
		end angle=180,
		x radius=10mm,
		y radius =2mm
		] ;
		%%%%%%%%%%%%%%%%%%%%%%%%%%%%%%%%%%%%%%%%%%%%%
		\draw [rotate=90,line width=1pt,black!50] (0.5,1) arc
		[
		start angle=0,
		end angle=-180,
		x radius=5mm,
		y radius =1mm
		] ;
		\draw [rotate=90,dashed, line width=1pt,black!50] (0.5,1) arc
		[
		start angle=0,
		end angle=180,
		x radius=5mm,
		y radius =1mm
		] ;
		%%%%%bAJAS
		%\draw [line width=1pt] (0,-4) ellipse (20mm and 4mm);
		%%%%%   
		\draw [rotate=90,line width=1pt,black!50] (2,-4) arc
		[
		start angle=0,
		end angle=-180,
		x radius=20mm,
		y radius =4mm
		] ;
		\draw [rotate=90,dashed, line width=1pt,black!50] (2,-4) arc
		[
		start angle=0,
		end angle=180,
		x radius=20mm,
		y radius =4mm
		] ;
		%%%%%%%%%%%
		\draw [rotate=90,line width=1pt,black!50] (1.5,-3) arc
		[
		start angle=0,
		end angle=-180,
		x radius=15mm,
		y radius =3mm
		] ;
		\draw [rotate=90,dashed, line width=1pt,black!50] (1.5,-3) arc
		[
		start angle=0,
		end angle=180,
		x radius=15mm,
		y radius =3mm
		] ;
		%%%%%%%%%%%
		\draw [rotate=90,line width=1pt,black!50] (1,-2) arc
		[
		start angle=0,
		end angle=-180,
		x radius=10mm,
		y radius =2mm
		] ;
		\draw [rotate=90,dashed, line width=1pt,black!50] (1,-2) arc
		[
		start angle=0,
		end angle=180,
		x radius=10mm,
		y radius =2mm
		] ;
		%%%%%%%%%%%%%%%%%%%%%%%%%%%%%%%%%%%%%%%%%%%%%
		\draw [rotate=90,line width=1pt,black!50] (0.5,-1) arc
		[
		start angle=0,
		end angle=-180,
		x radius=5mm,
		y radius =1mm
		] ;
		\draw [rotate=90,dashed, line width=1pt,black!50] (0.5,-1) arc
		[
		start angle=0,
		end angle=180,
		x radius=5mm,
		y radius =1mm
		] ;
%%%%%%%%%%%%%%%%%%%%%%%%%%%%
%%%%%%%%%%%%%%%
       %%%%%%%%%%%%%%%%%%%%%%%%%%%%%%%%%%%%%%%%%%%%%%%
		%%
		\draw [line width=1pt] (0,0) -- (-2,4);
		\draw [line width=1pt] (0,0) -- (2,4);
		\draw [line width=1pt] (0,0) -- (-2,-4);
		\draw [line width=1pt] (0,0) -- (2,-4);
		%%%%%%%%%%%%%%%%%
		\draw [line width=1pt,black!50] (0,0) -- (0,3.6);
		\draw [line width=1pt,black!50] (0,0) -- (0,-4.4);
		\draw [line width=1pt,black!50] (0,0) -- (1,3.68);
		\draw [line width=1pt,black!50] (0,0) -- (-1,3.68);
		\draw [line width=1pt,black!50] (0,0) -- (1,-4.37);
		\draw [line width=1pt,black!50] (0,0) -- (-1,-4.37);
		%%%%%%%%%%%%
		\draw [line width=1pt,black!50] (0,4) ellipse (20mm and 4mm);
		\draw [line width=1pt,black!50] (1.5,3) arc
		[
		start angle=0,
		end angle=-180,
		x radius=15mm,
		y radius =3mm
		] ;
		\draw [dashed, line width=1pt,black!50] (1.5,3) arc
		[
		start angle=0,
		end angle=180,
		x radius=15mm,
		y radius =3mm
		] ;
		%%%%%%%%%%%
		\draw [line width=1pt,black!50] (1,2) arc
		[
		start angle=0,
		end angle=-180,
		x radius=10mm,
		y radius =2mm
		] ;
		\draw [dashed, line width=1pt,black!50] (1,2) arc
		[
		start angle=0,
		end angle=180,
		x radius=10mm,
		y radius =2mm
		] ;
		%%%%%%%%%%%%%%%%%%%%%%%%%%%%%%%%%%%%%%%%%%%%%
		\draw [line width=1pt,black!50] (0.5,1) arc
		[
		start angle=0,
		end angle=-180,
		x radius=5mm,
		y radius =1mm
		] ;
		\draw [dashed, line width=1pt,black!50] (0.5,1) arc
		[
		start angle=0,
		end angle=180,
		x radius=5mm,
		y radius =1mm
		] ;
		%%%%%bAJAS
		%\draw [line width=1pt] (0,-4) ellipse (20mm and 4mm);
		%%%%%   
		\draw [line width=1pt,black!50] (2,-4) arc
		[
		start angle=0,
		end angle=-180,
		x radius=20mm,
		y radius =4mm
		] ;
		\draw [dashed, line width=1pt,black!50] (2,-4) arc
		[
		start angle=0,
		end angle=180,
		x radius=20mm,
		y radius =4mm
		] ;
		%%%%%%%%%%%
		\draw [line width=1pt,black!50] (1.5,-3) arc
		[
		start angle=0,
		end angle=-180,
		x radius=15mm,
		y radius =3mm
		] ;
		\draw [dashed, line width=1pt,black!50] (1.5,-3) arc
		[
		start angle=0,
		end angle=180,
		x radius=15mm,
		y radius =3mm
		] ;
		%%%%%%%%%%%
		\draw [line width=1pt,black!50] (1,-2) arc
		[
		start angle=0,
		end angle=-180,
		x radius=10mm,
		y radius =2mm
		] ;
		\draw [dashed, line width=1pt,black!50] (1,-2) arc
		[
		start angle=0,
		end angle=180,
		x radius=10mm,
		y radius =2mm
		] ;
		%%%%%%%%%%%%%%%%%%%%%%%%%%%%%%%%%%%%%%%%%%%%%
		\draw [line width=1pt,black!50] (0.5,-1) arc
		[
		start angle=0,
		end angle=-180,
		x radius=5mm,
		y radius =1mm
		] ;
		\draw [dashed, line width=1pt,black!50] (0.5,-1) arc
		[
		start angle=0,
		end angle=180,
		x radius=5mm,
		y radius =1mm
		] ;
		%%%%% 

		\end{scope}
		\end{tikzpicture}
	\\
		 \emph{Case d=2.}                         &                 &  \emph{Case d=4.} \\
	\end{tabular}
	\caption{\emph{Neighborhood of a singular Riemann surface.}}
	\label{fig:titulo}
\end{figure}

\begin{remark}
If $S$ is a compact singular Riemann surface, then it follows the following properties:
\begin{itemize}
    \item[\textbf{(1)}] The locus of singular points ${\rm Sing}(S)$ is a finite set.
    \item[\textbf{(2)}] The singular Riemann surface $S$ has many finitely irreducible components.
    \item[\textbf{(3)}] Each irreducible component $R$ of $S$ is a compact Riemann surface and  $\widetilde{R}$ is also a Riemann surface having finite punctures.
\end{itemize}
\end{remark}

\begin{definition}[\cite{Hidalgo-Reyes-Vega}*{Subsection 2.2}]
Let us consider two singular Riemann surfaces $S_1$ and $S_2$. An homeomorphism $f:S_{1}\to S_{2}$ is called \textbf{isomorphism} if it satisfies the following two properties:
\begin{itemize}
    \item[\textbf{(1)}] The map $f$ sends the locus singular points of $S_1$ onto the locus singular points of $S_2$, \emph{i.e.}, $f({\rm Sing(S_1)})={\rm Sing}(S_2)$.
    \item[\textbf{(2)}] The restriction $f:S_1\setminus {\rm Sing}(S_1)\to S_2\setminus {\rm Sing}(S_2)$ is holomorphic.
\end{itemize}
\end{definition}
Two singular Riemann surfaces $S_{1}$ and $S_{2}$ are \emph{isomorphic} if there exists an isomophism $f:S_{1}\to S_{2}$. An \emph{automorphism} of a singular Riemann surface $S$ is an isomorphism of $S$ to itself. The automorphisms set of a singular Riemann surface $S$, which will be denoted by ${\rm Aut}(S)$ has a group structure with the composition operation.

\subsection{Fiber product of Riemann surfaces}\label{Subsection:fiber_product_riemann_surface} 

Now, we introduce the concept of fiber product over Riemann surfaces. Also we explore the irreducible components associated to the fiber product and we define the normal fiber product.

\begin{definition}[\cite{Hidalgo-Reyes-Vega}*{Subsection 2.3}]
Let us fix three Riemann surfaces (not necessarily compact) $S_{0}$, $S_{1}$ and $S_{2}$ and two surjective holomorphic maps $\beta_{1}:S_{1}\to S_{0}$ and $\beta_{2}:S_{2}\to S_{0}$. The \textbf{fiber product} associated to the pairs $(S_{1},\beta_{1})$ and $(S_{2},\beta_{2})$ is defined as
\[
S_{1}\times_{(\beta_{1},\beta_{2})}S_{2}:=\left\{(z_{1},z_{2})\in S_{1}\times S_{2}: \beta_{1}(z_{1})=\beta_{2}(z_{2})\right\},
\] endowed with the topology induced by the product topology of $S_1\times S_2$.
There is associated a natural continuous map
$
\beta:S_1\times_{(\beta_1,\beta_2)}S_2\to S_0$, such that
\begin{equation}\label{eq:cont_proj_prod_fib}
\beta=\beta_1\circ\pi_1=\beta_2\circ\pi_2,
\end{equation}
where $\pi_j:S_1\times_{(\beta_1,\beta_2)}S_2\to S_j$ is the projection map $\pi_j(z_1,z_2)=z_j$, for $j\in\{1,2\}$.
\end{definition}

The fiber product of the pairs $(S_{1},\beta_{1})$ and $(S_{2},\beta_{2})$ enjoys the following universal property. If $S$ is a topological space and, there are surjective holomorphic maps $p_{1}:S\to S_{1}$ and $p_{2}:S\to S_{2}$ such that $\beta_{1}\circ p_{1}=\beta_{2}\circ p_{2}$, then there exists a unique continuous map $F:S\to S_{1}\times_{(\beta_1,\beta_2)}S_{2}$ such that in  diagram (\ref{eq:univeral-property}), it satisfies $p_j=\pi_{j}\circ F$, for $j\in\{1,2\}$.
\begin{equation}\label{eq:univeral-property}
\begin{tikzcd}
S
\arrow[bend left]{drr}{p_2}
\arrow[bend right,swap]{ddr}{p_1}
\arrow[dashed]{dr}[description]{F} & & \\
& S_{1}\times_{(\beta_1,\beta_2)}S_{2} \arrow{dr}{\beta}\arrow{r}{\pi_2} \arrow{d}[swap]{\pi_1}
& S_{2} \arrow{d}{\beta_2} \\
& S_{1} \arrow[swap]{r}{\beta_1}
& S_{0}
\end{tikzcd}
\end{equation}

The fiber product is, up to homeomorphisms, the unique topological space satisfying the above property. For more details of the fiber product of topological spaces, we refer the reader to \cite{Harris}*{p. 30}.

The following result assures that the fiber product is a singular Riemann surface.

\begin{proposition}[\cite{Hidalgo-Reyes-Vega}*{Proposition 2.2}] \label{pro:RSA}

The fiber product $S_{1}\times_{(\beta_{1},\beta_{2})}S_{2}$ is a singular Riemann surface. Moreover, if $p=(z_1, z_2)\in S_{1}\times_{(\beta_{1},\beta_{2})}S_{2}$ and, let $n_j$ be the local degree of $\beta_j$ at $z_j$ (for $j=1, 2$) and let $d$ be the greatest common divisor of $n_1$ and $n_2$, then the point $p$ is a singular point if and only if $d\geq 2$, in which case, it has a neighborhood of the form $V_{n_1,n_2}$.
\end{proposition}

The locus of singular points ${\rm Sing}(S_{1}\times_{(\beta_{1},\beta_{2})}S_{2})$ are pairs $(z_{1},z_{2})$, which $z_{i}$ is a ramification point of $S_{i}$, with $i\in\{1,2\}$.
%for which both $z_{1}$ is a ramification point of $S_{1}$ and $z_{2}$ is a ramification point of $S_{2}$.
On the other hand, the space 
\[
S_{1}\times_{(\beta_{1},\beta_{2})}S_{2}\setminus {\rm Sing}(S_1\times_{(\beta_1,\beta_2)}S_2),
\]
consists of a collection $\left\{\widetilde{R}_{\alpha}:\alpha\in\mathcal{A}\right\}$ of connected Riemann surfaces, for any family of indices $\mathcal{A}$. These surfaces satisfy the following properties:

\begin{itemize}
    \item[\textbf{(1)}] If the locus of singular points is a finite set, then the collection of Riemann surfaces $\left\{\widetilde{R}_{\alpha}:\alpha\in\mathcal{A}\right\}$ is finite.
    
    \item[\textbf{(2)}] The map $\beta$ (described in (\ref{eq:cont_proj_prod_fib})) restricts to $\widetilde{R}_{\alpha}$ not necessarily is a surjective holomorphic map.  
    
    \item[\textbf{(3)}] The Riemann surface $\widetilde{R}_{\alpha}$ has a collection of punctures, associated to the points in ${\rm Sing}(S_1\times_{(\beta_1,\beta_2)}S_2)$, and by filling in these points, we obtain a unique, up to biholomorphism, irreducible component $R_{\alpha}$  of the fiber product, and a surjective holomorphic map $\beta:R_{\alpha}\to S_0$ extending.     
\end{itemize}

\begin{definition}
The union of all the irreducible components is said the \textbf{normal fiber product}, which is denoted by $   \widetilde{S_1\times_{(\beta_1, \beta_2)} S_2}$, this is the normalization of the fiber product $S_1\times_{(\beta_1, \beta_2)} S_2$ when it is considered as a complex algebraic variety.
\end{definition}

 From the previous definition, there is a one-to-one correspondence between components of the normal fiber product and the irreducible components of the fiber product.

Remark that if $B\subset S_{0}$ is the discrete subset consisting of the union of the branch values of $\beta_1$ and $\beta_2$, and $S_0^*=S_0\setminus B$, $S_1^*=S_1\setminus \beta_1^{-1}(B)$ and $S_2^*=S_2\setminus \beta_2^{-1}(B)$, then
$$S_1^*\times_{(\beta_1, \beta_2)} S_2^*=\bigcup\limits_{\alpha\in\mathcal{A}}R_{\alpha}^*\subset \left(S_1\times_{(\beta_1, \beta_2)} S_2 \setminus {\rm Sing}(S_1\times_{(\beta_1,\beta_2)}S_2)\right)=\bigcup\limits_{\alpha\in\mathcal{A}}\widetilde{R}_{\alpha}\subset \bigcup\limits_{\alpha\in\mathcal{A}}R_{\alpha},$$ 
where $R_{\alpha}^*$ is a connected Riemann surface, this being complement in $\widetilde{R}_{\alpha}$ of a discrete set of points.

The following fact is consequence of the universal property of the fiber product and the above description.

\begin{lemma}[\cite{Hidalgo-Reyes-Vega}*{Lemma 2.5}]
If $S$ is a connected Riemann surface and $p_j:S\to S_j^*$ are surjective holomorphic maps so that $\beta_1\circ p_1=\beta_2\circ p_2$, then there exists an irreducible component $R_{\alpha}^*$  and a holomorphic map $h:S\to R_{\alpha}^*$ so that $p_j=\pi_j\circ h$, where $\pi_j:S_1\times_{(\beta_1,\beta_2)}S_2\to S_j$ is the projection map $\pi_j(z_1,z_2)=z_j$, for $j\in\{1,2\}$.
\end{lemma}

For the case that both holomorphic maps $\beta_{j}$ are of finite degree, the following result provides an upper bound on the number of irreducible components of the fiber product.

\begin{proposition}[\cite{Hidalgo-Reyes-Vega}*{Proposition 2.6}]
If $\beta_1$ and $\beta_2$ both have finite degree, then the number of irreducible components of the fiber product of the two pairs $(S_1, \beta_1)$ and $(S_2, \beta_2)$ is at most the greatest common divisor of the degrees of $\beta_1$ and $\beta_2$.
\end{proposition}

%%%%%%%%%%%%%%%%%%%%%%%%%%%%%%%%%%%%%%%%%%%%%%%%%%%%%%%%%%%%%%%%%%%%%%%%
%%%%%%%%%%%%%%%%%%%%%%%%%%%%%%%%%%%%%%%%%%%%%%%%%%%%%%%%%%%%%%%%%%%%%%%
%%%%%%%%%%%Section Ends space of fiber product of Riemann surfaces
%%%%%%%%%%%%%%%%%%%%%%%%%%%%%%%%%%%%%%%%%%%%%%%%%%%%%%%%%%%%%%%%%%%%%%%
%%%%%%%%%%%%%%%%%%%%%%%%%%%%%%%%%%%%%%%%%%%%%%%%%%%%%%%%%%%%%%%%%%%%%%%%%

\section{Ends space of the fiber product over Riemann surfaces}\label{sec:ends_spaces_fiber_product_riemann_surfaces}

Let  $S_{0}$, $S_{1}$ and $S_{2}$ be Riemann surfaces (not necessarily compact), let $\beta_{1}:S_{1}\to S_{0}$ and $\beta_{2}:S_{2}\to S_{0}$ be surjective holomorphic maps. The fiber product $S_{1}\times_{(\beta_{1},\beta_{2})}S_{2}$ is a subspace of the product space $S_{1}\times S_{2}$ inheriting the properties: Hausdorff, second countable, locally compact and locally connected. It follows that we can associate to each connected component of the fiber product $S_1\times_{(\beta_1, \beta_2)} S_2$ its respective ends space.  Recall that if the locus of singular points of $S_1\times_{(\beta_1, \beta_2)} S_2$ is finite, then the number of irreducible components $R_{1},\ldots, R_{k}$ of the fiber product is also finite, for any $k\in\mathbb{N}$. Given that such irreducible components are Riemann surfaces, the ends space of the normalization of the fiber product is
\[
 {\rm Ends}(\widetilde{S_1\times_{(\beta_1, \beta_2)} S_2})=\bigcup\limits_{i=1}^{k}{\rm Ends}(R_{i}).
\]

The following result describes  the ends space of a connected fiber product over Riemann surfaces denoted as ${\rm Ends}(S_{1}\times_{(\beta_{1},\beta_{2})}S_{2})$,  from the point of view of the ends space of its respective irreducible components, when its locus of singular points is finite. 

\begin{theorem}\label{teo1}
 Suppose that the fiber product $S_1\times_{(\beta_1, \beta_2)} S_2$ is connected and its locus of singular points is finite, then ${\rm Ends}(S_1\times_{(\beta_1, \beta_2)} S_2)$ is homeomorphic to ${\rm Ends}(\widetilde{S_1\times_{(\beta_1, \beta_2)} S_2})$.
\end{theorem}

The following remark is necessary to the proof of the result above.

\begin{remark}\label{ob3.2} 
For the fiber product $S_{1}\times_{(\beta_{1},\beta_{2})}S_{2}$ the following properties holds:

\begin{itemize}
    \item[\textbf{(1)}] As the locus set of singular points ${\rm Sing}(S_1\times_{(\beta_1, \beta_2)} S_2)$ is a finite compact space, the subset 
\begin{equation*}\label{eq:singular_irreducible_componnent}
{\rm Sing}(R_{i}):=R_{i}\setminus \widetilde{R}_{i}\subset R_{i},
\end{equation*}
is a finite compact space, for each $i\in\{1,\ldots,k\}$.

   \item[\textbf{(2)}] If we consider the end $[U_{n}]_{n\in\mathbb{N}}$ in the irreducible component $R_{i}$ and the compact space ${\rm Sing}(R_{i})$, using  Definition \ref{definition:end}, there exists $s\in\mathbb{N}$ such that ${\rm Sing}(R_{i})\cap U_{s}=\emptyset$. We can suppose without loss of generality that $s=1$. This implies that ${\rm Sing}(R_{i})\cap U_{n}=\emptyset$ or, equivalent $U_{n}\subset \widetilde{R}_{i}=R_{i}\setminus {\rm Sing}(R_{i})$, for all $n\in\mathbb{N}$.
   
   \item[\textbf{(3)}] Given that the Riemann surface $\widetilde{R}_{i}$ is a subspace of the fiber product $S_1\times_{(\beta_1, \beta_2)} S_2$, then the infinite nested sequence $(U_{n})_{n\in\mathbb{N}}$ of $\widetilde{R}_{i}$ is also an infinite nested sequence of the fiber product $S_1\times_{(\beta_1, \beta_2)} S_2$, which defines the \textbf{common end} (this nomenclature appears in \cite{RamVal}*{Remark 3.6}) 
\begin{equation*}\label{eq:map_F}
[\widetilde{U}_{n}]_{n\in\mathbb{N}}\in {\rm Ends}(S_1\times_{(\beta_1, \beta_2)} S_2),
\end{equation*}
such that the sequence $(U_{n})_{n\in\mathbb{N}}$ belongs to the class $[\widetilde{U}_{n}]_{n\in\mathbb{N}}$. 

   \item[\textbf{(4)}] If the infinite nested sequences $(V_{n})_{n\in\mathbb{N}}$ and $(U_{n})_{n\in\mathbb{N}}$ of $R_{i}$ are equivalent, then by the above construction, these infinite sequences define their respective common ends $[\widetilde{V}_{n}]_{n\in\mathbb{N}}$ and $[\widetilde{U}_{n}]_{n\in\mathbb{N}}$ in the fiber product  $S_1\times_{(\beta_1, \beta_2)} S_2$ such that $[\widetilde{V}_{n}]_{n\in\mathbb{N}}=[\widetilde{U}_{n}]_{n\in\mathbb{N}}$.
\end{itemize}

\end{remark}

With this observation we can star the proof of Theorem \ref{teo1}.

\begin{proof}
Let us star defining the map
\[
F:\bigcup\limits_{i=1}^{k}{\rm Ends}(R_{i})\to {\rm Ends}(S_1\times_{(\beta_1, \beta_2)} S_2),
\]
such that it sends the end $[U_{n}]_{n\in\mathbb{N}}\in \bigcup\limits_{i=1}^{k}{\rm Ends}(R_{i})$ to the common end $[\widetilde{U}_{n}]_{n\in\mathbb{N}}\in {\rm Ends}(S_1\times_{(\beta_1, \beta_2)} S_2)$ (see \textbf{(3)} in Remark \ref{ob3.2}). By \textbf{(4)} in  Remark \ref{ob3.2}, it follows that $F$ is a well-defined map.\\

\noindent \textbf{Injectivity.}  Let us consider the different ends $[U_{n}]_{n\in\mathbb{N}}$ and $[V_{n}]_{n\in\mathbb{N}}$ in $\bigcup\limits_{i=1}^{k}{\rm Ends}(R_{i})$. We must prove that $[\widetilde{U}_{n}]_{n\in\mathbb{N}}=F([U_{n}]_{n\in\mathbb{N}})$ is different to $ F([V_{n}]_{n\in\mathbb{N}})=[\widetilde{V}_{n}]_{n\in\mathbb{N}}$.
We study the following cases:

\vspace{1.5mm}

\noindent \emph{Case 1.} The ends are in different irreducible components, \emph{i.e.}, there are $k_{1}\neq  k_{2}\in\{1,\ldots,k\}$ such that $[U_{n}]_{n\in\mathbb{N}}\in {\rm Ends}(R_{k_1})$ and $[V_{n}]_{n\in\mathbb{N}}\in {\rm Ends}(R_{k_2})$. From \textbf{(1)} in Remark \ref{ob3.2} we hold that ${\rm Sing}(R_{k_i})$
is a finite compact subset of $R_{k_{i}}$, for each $i\in\{1,2\}$. Using \textbf{(2)} in the preceding remark we obtain that for all $n\in\mathbb{N}$,
\[
{\rm Sing}(R_{k_1})\cap U_{n}=\emptyset\quad \text{and}\quad {\rm Sing}(R_{k_2})\cap V_{n}=\emptyset.
\]
By hypothesis the classes $[U_{n}]_{n\in\mathbb{N}}$ and $[V_{n}]_{n\in\mathbb{N}}$ are different, using Remark \ref{remarK:different-classes} we obtain that $U_{1}\cap V_{1}=\emptyset$. As the infinite nested sequences $(U_{n})_{n\in\mathbb{N}}$ and $(V_{n})_{n\in\mathbb{N}}$ are belonged to the class $[\widetilde{U}_{n}]_{n\in\mathbb{N}}$ and $[\widetilde{V}_{n}]_{n\in\mathbb{N}}$ of ${\rm Ends}(S_1\times_{(\beta_1, \beta_2)} S_2)$, respectively, then there are  $m,l\in\mathbb{N}$ such that $\widetilde{U}_{l}\cap \widetilde{V}_{m}=\emptyset$.

\vspace{1.5mm}
\noindent \emph{Case 2.} The ends are in the same irreducible component, \emph{i.e.}, there is $k'\in\{1,\ldots,k\}$ such that the ends $[U_{n}]_{n\in\mathbb{N}},[V_{n}]_{n\in\mathbb{N}}\in{\rm Ends}(R_{k'})$. Since these ends are different, from Remark \ref{remarK:different-classes} it follows that
\begin{equation}\label{eq:intersection_ends}
U_{1}\cap V_{1}=\emptyset.   
\end{equation}
From \textbf{(1)} in Remark \ref{ob3.2} we hold that ${\rm Sing}(R_{k'})$
is a finite compact subset of $R_{k'}$. Using \textbf{(2)} in the preceding remark we obtain that for all $n\in\mathbb{N}$,
\begin{equation}\label{eq:intersection_singular}
{\rm Sing}(R_{k'})\cap U_{n}=\emptyset\quad \text{and}\quad {\rm Sing}(R_{k'})\cap V_{n}=\emptyset.
\end{equation}
As the infinite nested sequences $(U_{n})_{n\in\mathbb{N}}$ and $(V_{n})_{n\in\mathbb{N}}$ are belonged to the class $[\widetilde{U}_{n}]_{n\in\mathbb{N}}$ and $[\widetilde{V}_{n}]_{n\in\mathbb{N}}$ of ${\rm Ends}(S_1\times_{(\beta_1, \beta_2)} S_2)$, respectively, from equations  (\ref{eq:intersection_ends}) and (\ref{eq:intersection_singular}) it follows that there are $l,m\in\mathbb{N}$ such that $\widetilde{U}_{l}\cap\widetilde{V}_{m}=\emptyset$. It proves that the map $F$ is injective.\\

\noindent \textbf{Surjectivity.} We take an end $[\widetilde{U}_{n}]_{n\in\mathbb{N}}$ of $ S_1\times_{(\beta_1, \beta_2)} S_2$, then we must prove that there exists at least one end   $[U_{n}]_{n\in\mathbb{N}}\in {\rm Ends}(R_{k'})$, for any  $k'\in\{1,\ldots,k\}$, such that $F([U_{n}]_{n\in\mathbb{N}})=[\widetilde{U}_{n}]_{n\in\mathbb{N}}$. From Definition \ref{definition:end}, we can assume that ${\rm Sing}(S_1\times_{(\beta_1, \beta_2)} S_2)\cap \widetilde{U}_{n}=\emptyset$, for all $n\in\mathbb{N}$. As the open subset $\widetilde{U}_{n}\subset S_1\times_{(\beta_1, \beta_2)} S_2$ is connected, then the infinite nested sequence  $(\widetilde{U}_{n})_{n\in\mathbb{N}}$ belongs to any connected component of
\[
S_1\times_{(\beta_1, \beta_2)} S_2/{\rm Sing}(S_1\times_{(\beta_1, \beta_2)} S_2)=\bigcup\limits_{i=1}^{k} \widetilde{R}_{i},
\]
it means, there is $k'\in\{1,\ldots,k\}$ such that $(\widetilde{U}_{n})_{n\in\mathbb{N}}$ is an infinite nested sequence of $\widetilde{R}_{k'}$, which defines the end $[U_{n}]$ of $R_{k'}$ such that $(\widetilde{U}_{n})_{n\in\mathbb{N}}$ is in the class $[U_{n}]_{n\in\mathbb{N}}$ (see \textbf{(3)} in Remark \ref{ob3.2})). Finally, by the definition of the map $F$ we obtain $F([U_{n}]_{n\in\mathbb{N}})=[\widetilde{U}_{n}]_{n\in\mathbb{N}}$.\\

\noindent \textbf{Continuity.} We consider an end $[U_{n}]_{n\in\mathbb{N}}\in {\rm Ends}(R_{k'})$, for any $k'\in\{1,\ldots,k\}$,  and the open subset $W^{\ast}$ of ${\rm Ends}(S_1\times_{(\beta_1, \beta_2)} S_2)$ such that $F([U_{n}]_{n\in\mathbb{N}})=[\widetilde{U}_{n}]_{n\in\mathbb{N}}$, for any non-empty connected open subset $W$ of $ S_1\times_{(\beta_1, \beta_2)} S_2$, with compact boundary. Then we must prove that there is an open subset $V^{\ast}\subset {\rm Ends}(R_{k'})$ having the end $[U_{n}]_{n\in\mathbb{N}}$ such that $F(V^{\ast})\subset W^{\ast}$. 

From the surjectivity of the map $F$, it follows that $(\widetilde{U}_{n})_{n\in\mathbb{N}}$ is an infinite nested sequence of $R_{k'}$, which is belonged to the class $[U_{n}]_{n\in\mathbb{N}}\in R_{k'}$, for any $k'\in\{1,\ldots,k\}$. By hypothesis, the class $[\widetilde{U}_{n}]_{n\in\mathbb{N}}$ is in the open $W^{\ast}$, it means that $\widetilde{U}_{m}\subset W$ for any  $m\in\mathbb{N}$, as the infinite nested sequence $(\widetilde{U}_{n})_{n\in\mathbb{N}}$ is in the class $[U_{n}]_{n\in\mathbb{N}}$, then there exists a natural number $p(m)\in\mathbb{N}$ such that $W\supset \widetilde{U}_{m}\supset U_{p(m)}$. Hence,  the connected open subset $U^{\ast}_{p(m)}\subset{\rm Ends}(R_{k'})$ contains the end $[U_{n}]_{n\in\mathbb{N}}$ and satisfies that $F(U^{\ast}_{p(m)})\subset W^{\ast}$.\\

As the continuous map $F$ is defined from a compact space to a Hausdorff space, from the Theorem 2.1 in \cite{Dugu}*{p. 226} it follows that the function map $F$ is closed. Thereby, we conclude that $F$ is a homeomorphism. 
\end{proof}

\subsection{Branched covering maps}

We shall consider the fiber product coming from covering branched maps and we shall establish conditions on such maps guaranteeing connectedness of the fiber product. Recall that a topological surface is of \emph{infinite-type} if it has fundamental group infinitely generated.

\vspace{1.5mm}
Let $S_{1}$ and $S_{2}$ be Riemann surfaces such that $S_{1}$ is of infinite-type and, $S_2$ has $p\geq 2$ ends and genus $g\in\mathbb{N}_{0}\cup\{\infty\}$. Let $\beta_{i}:S_{i}\to\mathbb{C}$  be a branched covering map, for each $i\in\{1,2\}$ such that:

\begin{enumerate}
    \item[\textbf{(1)}] Their branched points are the discrete subsets $A_{1}$ and $A_{2}$ of $\mathbb{C}$ respectively, such that $A_{2}\subset A_{1}$ and $A_{1}$ is infinite.
    \item[\textbf{(2)}] For each $z \in \mathbb{C}\setminus A_{i}$, the fiber $\beta^{-1}_{i}(z)$ consists of $p_{i}\geq 2$ elements.
    \item[\textbf{(3)}] The maps $\beta_{1}$ and $\beta_{2}$ are proper,  that is, the inverse image of any compact subset $K \subset \mathbb{C}$ under $\beta_{i}$ is also a compact subset of $S_{i}$. 
\end{enumerate}

\begin{remark}\label{remarK:projection_map_beta_proper}
Given that the branched covering maps $\beta_{1}$ and $\beta_{2}$ are proper, then the continuous projection map $\beta:S_{1}\times_{(\beta_{1},\beta_{2})}S_{2}\to\mathbb{C}$ described in equation (\ref{eq:cont_proj_prod_fib}) is also proper, because $\beta^{-1}(B)$ is a closed subset of the compact $\beta^{-1}_{1}(B)\times\beta^{-1}_{2}(B)$, for each compact subset $B\subset \mathbb{C}$. Therefore, the inverse image $\beta^{-1}(B)$ is compact.
\end{remark}

The following result guarantees the connectedness of fiber product $S_{1}\times_{(\beta_{1},\beta_{2})}S_{2}$. 

\begin{theorem}\label{theorem:ends_space_of_fiber_product}
Suppose that the locus of singular points of the fiber product $S_{1}\times_{(\beta_{1},\beta_{2})}S_{2}$ is
\[
{\rm Sing}(S_{1}\times_{(\beta_{1},\beta_{2})}S_{2})=\{(z_{1},z_{2})\in S_{1}\times_{(\beta_{1},\beta_{2})}S_{2}: \beta_{1}(z_{1})=\beta_{2}(z_2)\in A_{2}\}.
\]
Then the fiber product
 $S_{1}\times_{(\beta_{1},\beta_{2})}S_{2}$ is connected.
\end{theorem}

\begin{proof}
Let us star by embedding of $p_2$ different ways the Riemann surface $S_{1}$ in the fiber product $S_{1}\times_{(\beta_{1},\beta_{2})}S_{2}$. The embedding is given by $J_{k}:S_{1} \hookrightarrow S_{1}\times_{(\beta_{1},\beta_{2})}S_{2}$, for each $k\in\{1,\ldots,p_{2}\}$. We fix $z_{0}\in S_{1}\setminus \beta_{1}^{-1}(A_{2})$ which we will call the \emph{base point} and consider the complex number $u$ such that $u\notin A_{2}$ and $\beta_{1}(z_{0})=u$. Given that the order of the covering function $\beta_{2}$ is $p_{2}$ and the complex $u$ is not one of its branch points, then the fiber $\beta_{2}^{-1}(u)\subset S_{2}$ is conformed by $p_{2}$ points $w_{1},\ldots,w_{p_{2}}$. Now, let us consider the element $w_{k}$, for $k\in\{1,\ldots,p_{2}\}$, then the embedding  $J_{k}$ sends the base point $z_{0}$ in the couple $(z_{0},w_{k})\in S_{1}\times_{(\beta_{1},\beta_{2})}S_{2}$, 
\begin{equation}\label{eq:embeding-based-point}
 z_{0}\mapsto (z_{0},w_{k}).  
\end{equation}

Taking a point $z\in S_{1}\setminus\beta_{1}^{-1}(A_{2})$ we will define on the fiber product the couple at which is sent $z$ under the embedding $J_{k}$.  We denote by $v$ the complex number such that $v\notin A_{2}$ and $\beta_{1}(z)=v$. As $v$ is not a branch point of $\beta_{2}$, then the fiber $\beta_{2}^{-1}(v)\subset S_{2}$  is conformed by $p_{2}$ points. Let us consider the curve $\gamma$ from the closed interval $[0,1]$ to $\mathbb{C}\setminus A_{1}$ with end points $u$ and $v$, that is $\gamma(0)=u$ and $\gamma(1)=v$. If we consider the lifting $\tilde{\gamma}$ in $S_{2}$ of the curve $\gamma$ such that one of its endpoints is $w_{k}=\tilde{\gamma}(0)$ (see equation (\ref{eq:embeding-based-point})) and the other endpoint $\tilde{\gamma}(1)$ belongs to the fiber $\beta_{2}^{-1}(v)$. Then, the embedding $J_{k}$ sends the point $z$ into the couple $(z,\tilde{\gamma}(1))\in S_{1}\times_{(\beta_{1},\beta_{2})}S_{2}$,

\begin{equation}\label{eq:embeding-based-point-2}
 z\mapsto (z,\tilde{\gamma}(1)).  
\end{equation}
Finally, the embedding $J_{k}$ sends each point
$z\in \beta_{1}^{-1}(A_{1})$ into the unique couple
\[
(z,w)\in {\rm Sing}(S_{1}\times_{(\beta_{1},\beta_{2})}S_{2}),
\]
such that $\beta_{1}(z)=\beta_{2}(w)\in A_{2}$. By construction and using the analytic continuation we have that map  $J_{k}$ is a holomorphic embedding, i.e., $J_{k}(S_{1})=R_{k}$ is biholomorphic to $S_{1}$. Additionally, we have
\[
\bigcap_{k=1}^{p_2} R_{k}={\rm Sing}(S_{1}\times_{(\beta_{1},\beta_{2})}S_{2}).
\]
We note that the singular Riemann surface $S_{1}\times_{(\beta_{1},\beta_{2})}S_{2}$ can be written as the union 
\begin{equation*}
S_{1}\times_{(\beta_{1},\beta_{2})}S_{2}= \bigcup\limits_{k=1}^{p_2} R_{k}, 
\end{equation*}
because if we consider the pair $(z,w)$ in $S_{1}\times_{(\beta_{1},\beta_{2})}S_{2}$ such that $\beta_{1}(z)=\beta_{2}(w)\notin A_{2}$, then the point $w$ is belonged to fiber $\beta_{2}^{-1}(\beta_{1}(z))$, which is conformed by $p_{2}$ points $w_{1},\ldots,w_{p_{2}}$, it implies that there exist $k\in\{1,\ldots,p_{2}\}$ such that $w=w_{k}$, thus we conclude that $(z,w)=(z,w_{k})\in R_{k}$. Contrary, if $\beta_{1}(z)=\beta_{2}(w)\in A_{2}$, then 
\[
(z,w)\in {\rm Sing}(S_{1}\times_{(\beta_{1},\beta_{2})}S_{2})\subset \bigcap\limits_{k=1}^{p_2} R_{k}\subset \bigcup\limits_{k=1}^{p_2} R_{k}.
\]
Given that the fiber product $S_{1}\times_{(\beta_{1},\beta_{2})}S_{2}$ is the union of $p_{2}$ connected spaces having non empty intersection, then we conclude that $S_{1}\times_{(\beta_{1},\beta_{2})}S_{2}$ is connected.
\end{proof}

\begin{corollary}\label{corollary:properties_irreducible_components2}
The space coming from the  fiber product $S_{1}\times_{(\beta_{1},\beta_{2})}S_{2}$ by removing its locus of singular points $ {\rm Sing }(S_{1}\times_{(\beta_{1},\beta_{2})}S_{2}$) is conformed by $p_{2}$ connected components $\widetilde{R}_{1},\ldots,\widetilde{R}_{p_{2}}$ such that 
\[
\widetilde{R}_{k}=R_{k}\setminus {\rm  Sing}(S_{1}\times_{(\beta_{1},\beta_{2})}S_{2}),
\]
is topologically equivalent to $S_{1}$ removing a discrete set of cardinality ${\rm card}(A_{2})$, for each $k\in\{1,\ldots,p_2\}$. Moreover, the irreducible component associated $R_{k}$  to $\widetilde{R}_{k}$ is biholomorphically equivalent to the Riemann surface $S_{1}$.
\end{corollary}

Now, we shall describe the relationship of the ends spaces of the fiber product $S_{1}\times_{(\beta_{1},\beta_{2})}S_{2}$ and the topology type of the Riemann surfaces $S_{1}$ and $S_{2}$. It is necessary to introduce the following concept. 

\vspace{1.5mm}
\begin{definition}
We say that the fiber $\beta_{1}^{-1}(A_{2})$ \textbf{does not accumulate at the end}  $[V_{n}]_{n\in\mathbb{N}}$ of $S_{1}$ if there is $s\in\mathbb{N}$ such that $V_{s}\cap \beta^{-1}_{1}(A_{2})=\emptyset$. In this case, it can be suppose without loss of generality that $s=1$. This implies that $V_{n}\cap \beta_{1}^{-1}(A_{2})=\emptyset$, for each $n\in\mathbb{N}$. We denote by $\mathcal{E}$ the subset of ${\rm Ends}(S_{1})$ of all ends such the fiber $\beta_{1}^{-1}(A_{2})$ not accumulate. 
\end{definition}

\begin{remark}\label{remark:embedding_ends_not_accumulate}
The embedding $J_{k}$ described in equation (\ref{eq:embeding-based-point-2}) and the subset $\mathcal{E}\subset {\rm Ends}(S_{1})$ define the following properties in the ends space of the fiber product $S_{1}\times_{(\beta_{1},\beta_{2})}S_{2}$:

\begin{itemize}
\item[\textbf{(1)}] If we consider the end $[V_{n}]_{n\in\mathbb{N}}\in\mathcal{E}$, then the image of sequence $(V_{n})_{n\in\mathbb{N}}$ under the embedding $J_{k}$ is also an infinite nested sequence $
(V_{n}^{k})_{n\in\mathbb{N}}$ of the fiber product $S_{1}\times_{(\beta_{1},\beta_{2})}S_{2}$, where $J_{k}(V_{n}):=V_{n}^{k}$ and $k\in\{1,\ldots, p_{2}\}$. Such sequence defines the \textbf{common end} $[V_{n}^{k}]_{n\in\mathbb{N}}$ in $S_{1}\times_{(\beta_{1},\beta_{2})}S_{2}$.

\item[\textbf{(2)}] For each $k\in\{1,\ldots,p_{2}\}$, the map $J_{k}$ induces an embedding for \textbf{common ends}
\[
\tilde{J}_{k}:\mathcal{E}\hookrightarrow {\rm Ends}(S_{1}\times_{(\beta_{1},\beta_{2})}S_{2}).
\]

\item[\textbf{(3)}] Given that
\[
\bigcap_{k=1}^{p_2} J_{k}(S_{1})=\bigcap_{k=1}^{p_2} R_{k}={\rm Sing}(S_{1}\times_{(\beta_{1},\beta_{2})}S_{2}),
\]
then
\[
\bigsqcup_{k=1}^{p_{2}}\tilde{J}_{k}(\mathcal{E})\subset {\rm Ends}(S_{1}\times_{(\beta_{1},\beta_{2})}S_{2}).
\]
\end{itemize}
\end{remark}

\begin{theorem}\label{theorem:ends_space_of_fiber_product2}
Suppose that the locus of singular points of the fiber product $S_{1}\times_{(\beta_{1},\beta_{2})}S_{2}$ is
\[
{\rm Sing}(S_{1}\times_{(\beta_{1},\beta_{2})}S_{2})=\{(z_{1},z_{2})\in S_{1}\times_{(\beta_{1},\beta_{2})}S_{2}: \beta_{1}(z_{1})=\beta_{2}(z_2)\in A_{2}\}.
\]
Then it holds the following properties:

\begin{itemize}
  \item[\textbf{(1)}] If the subset $A_{2}\subset \mathbb{C}$ is finite, then the ends space of the fiber product $S_{1}\times_{(\beta_{1},\beta_{2})}S_{2}$ can be written as the union of $p_{2}$ copies of ${\rm Ends}(S_{1})$. In other words, the spaces
  \[
  {\rm Ends}(S_{1}\times_{(\beta_{1},\beta_{2})}S_{2}) \quad \text{ and } \quad \bigsqcup_{k=1}^{p_2}\tilde{J}_{k}({\rm Ends}(S_{1}))
  \]
  are homeomorphic.
 
 \item[\textbf{(2)}] If $A_{2}$ is an infinite set and the subset ${\rm Ends}(S_{1})\setminus \mathcal{E}$ is  conformed by only  one end, then the ends space of the fiber product $S_{1}\times_{(\beta_{1},\beta_{2})}S_{2}$ is
 \[
 {\rm Ends}(S_{1}\times_{(\beta_{1},\beta_{2})}S_{2})=\{[\tilde{U}_{n}]_{n\in\mathbb{N}}\}\cup \left(\bigsqcup_{k=1}^{p_2}\tilde{J}_{k}(\mathcal{E})\right).
 \]
\end{itemize}
\end{theorem}

Let us remark that the property \textbf{(2)} is an equality in the category of sets, not in the category of topological spaces. The proof of property \textbf{(1)} follows from  Theorem \ref{teo1} and Corollary \ref{corollary:properties_irreducible_components2}. Let us proceed with the proof of property \textbf{(2)}:

\begin{proof}

 \textbf{Step 1.} We shall build a suitable increasing sequence  of compact $B_{1}\subset B_{2}\subset \ldots $ covering the fiber product $S_{1}\times_{(\beta_{1},\beta_{2})}S_{2}$. 
\\

Recall that each Riemann surface is a $\sigma$-compact space, then there is an increasing sequence of compact subsets $K_{1}\subset K_{2}\subset \ldots$ of $S_{1}$ covering the surface $S_{1}$ \emph{i.e.}, 
\[
\bigcup\limits_{n\in\mathbb{N}}K_{n}=S_{1}.
\]
Now, as the subset $A_{2}$ of $\mathbb{C}$ is an infinite subset of $A_{1}$, then the fiber $\beta_{1}^{-1}(A_{2})$ must be (countable) infinite and it is written as $\beta_{1}^{-1}(A_{2})=\{\textbf{w}_{n}:n\in\mathbb{N}\}$. So, for the only end $[U_{n}]_{n\in\mathbb{N}}$ of $S_{1}$ not belonged to $\mathcal{E}$ one can suppose without of generality that:
\begin{itemize}
\item[$\ast$] The connected open subset $U_{n}\subset S_{1}$ is a connected component of $S_{1}\setminus K_{n}$, such that its boundary $\partial U_{n} \subset \partial K_{n}$, for each $n\in \mathbb{N}$.    

\item[$\ast$] The points of the set $\beta_{1}^{-1}(A_{2})=\{\textbf{w}_{n}:n\in\mathbb{N}\}$ satisfies for each $n\in\mathbb{N}$ that

$$\{\textbf{w}_{1},\ldots,\textbf{w}_{n-1}\}\subset  K_{n} \quad \text{and} \quad \{\textbf{w}_{n},\textbf{w}_{n+1},\ldots\}\subset U_{n}.$$
\end{itemize}

Now, for each $n\in\mathbb{N}$ we consider the  subset $K_{n}\setminus \{\textbf{w}_{1},\ldots,\textbf{w}_{n-1}\}$ of $S_{1}$, then the closure in the fiber product $S_{1}\times_{(\beta_{1},\beta_{2})}S_{2}$ of the image $J_{k}(K_{n}\setminus \{\textbf{w}_{1},\ldots,\textbf{w}_{n-1}\})$ denoted by
\[
\overline{K}_{n}^{k}\subset S_{1}\times_{(\beta_{1},\beta_{2})}S_{2},
\]
is a compact subset, for each $k\in\{1,\ldots,p_{2}\}$. Thus, for each $n\in\mathbb{N}$, we get the new compact subset $B_{n}$ of the fiber product $S_{1}\times_{(\beta_{1},\beta_{2})}S_{2}$, where
 \begin{equation}\label{eq:special_compact}
    B_{n}:=\bigcup\limits_{k=1}^{p_{2}}\overline{K}_{n}^{k}\subset S_{1}\times_{(\beta_{1},\beta_{2})}S_{2}.
    \end{equation}
It satisfies that the intersection
 \[
B_{n}\bigcap {\rm Sing}(S_{1}\times_{(\beta_{1},\beta_{2})}S_{2})=\bigcup_{k=1}^{n-1}\beta_{1}^{-1}(\textbf{w}_{k}).
\]
Therefore, the family of compact  $\{B_{n}:n\in\mathbb{N}\}$ defines the increasing sequence $B_{1}\subset B_{2}\subset \ldots$ covering the fiber product, \emph{i.e.}, 
    \[
    S_{1}\times_{(\beta_{1},\beta_{2})}S_{2}=\bigcup\limits_{n\in\mathbb{N}}B_{n}.
    \]
    In other words, the fiber product is a $\sigma$-compact space.\\

\textbf{Step 2.} Now, we shall build the end $[\tilde{U}_{n}]_{n\in\mathbb{N}}$ of the fiber product $S_{1}\times_{(\beta_{1},\beta_{2})}S_{2}$, which is described in the theorem. \\

Given that the fiber product  $S_{1}\times_{(\beta_{1},\beta_{2})}S_{2}$ is connected locally and compact locally, then for each $n\in\mathbb{N}$ the set
\[
S_{1}\times_{(\beta_{1},\beta_{2})}S_{2}\setminus B_{n}
\]
has a finite number of connected components $U^{n}_{1},\ldots, U^{n}_{i(n)}$, for any $i(n)\in\mathbb{N}$, whose boundaries are compact. 

\begin{remark}\label{remark:open_disjoint_Ends_product_fiber}
The connected components of $S_{1}\times_{(\beta_{1},\beta_{2})}S_{2}\setminus B_{n}$ have the following properties:

\begin{itemize}
\item[\textbf{(1)}] We can suppose without of generality that each one of them is non-compact. 

\item[\textbf{(2)}] One of these open connected components must contain the union
\[
\bigsqcup\limits_{k=1}^{p_{2}}J_{k}(U_{n}\setminus\{\textbf{w}_{l}:l\in\mathbb{N}\}),
\]
for each $n\in\mathbb{N}$. Such connected component is denoted by
\begin{equation*}\label{eq:special-ends}
\tilde{U}_{n}:=U^{n}_{1}.
\end{equation*}

\item[\textbf{(3)}] Given that $B_{n}\subset B_{n+1}$ and from the choice of the connected component $\tilde{U}_{n}$, we obtain that  $\tilde{U}_{n}\supset \tilde{U}_{n+1}$, for each $n\in\mathbb{N}$.

\item[\textbf{(4)}] In particular, for $n=1$ the connected components $U_{1}^{1},\ldots,U_{i(1)}^{1}$ of $S_{1}\times_{(\beta_{1},\beta_{2})}S_{2}\setminus B_{1}$ define the pair disjoint open subsets $(U_{1}^{1})^{\ast},\ldots,(U_{i(1)}^{1})^{\ast}$  of ${\rm Ends}(S_{1}\times_{(\beta_{1},\beta_{2})}S_{2})$, respectively. From the definition of the embedding $J_{k}$ we hold the equality 
\[
\left(\bigsqcup_{k=1}^{p_2}\tilde{J}_{k}(\mathcal{E})\right)=\bigsqcup\limits_{j=2}^{i(1)} (U_{j}^{1})^{\ast}.
\]

\item[\textbf{(5)}] From the construction of $\tilde{U}_{n}$, we hold that the open subset $(\tilde{U}_{n})^{\ast}$ of ${\rm Ends}(S_{1}\times_{(\beta_{1},\beta_{2})}S_{2})$ satisfies
\[
(\tilde{U}_{n})^{\ast} \cap \left(\bigsqcup_{k=1}^{p_2}\tilde{J}_{k}(\mathcal{E})\right)=\emptyset.
\]
\end{itemize}
\end{remark}

From \textbf{(1)}, \textbf{(2)} and \textbf{(3)} in the previous remark, we hold an infinite nested sequence $(\tilde{U}_{n})_{n\in\mathbb{N}}$ of non-empty connected open subsets of the fiber product $S_{1}\times_{(\beta_{1},\beta_{2})}S_{2}$, which defines the  \textbf{secret end} $[\tilde{U}_{n}]_{n\in\mathbb{N}}$ of $S_{1}\times_{(\beta_{1},\beta_{2})}S_{2}$ (this nomenclature appears in \cite{RamVal}*{p. 559}). By \textbf{(4)} and \textbf{(5)} in Remark \ref{remark:open_disjoint_Ends_product_fiber}, we hold the following properties of the secret end $[\tilde{U}_{n}]_{n\in\mathbb{N}}$:

\begin{itemize}
\item[\textbf{(a)}] It is not in the disjoint union $\bigsqcup\limits_{k=1}^{p_2}\tilde{J}_{k}(\mathcal{E})$.

\item[\textbf{(b)}] The secret end $[\tilde{U}_{n}]_{n\in\mathbb{N}}$ is in the open subset $\tilde{U}_{n}^{\ast} \subset {\rm Ends}(S_{1}\times_{(\beta_{1},\beta_{2})}S_{2})$, for each $n\in\mathbb{N}$.
\end{itemize}

\vspace{4mm}
\textbf{Step 3.} Finally, we shall prove the equality 
\[
{\rm Ends}(S_{1}\times_{(\beta_{1},\beta_{2})}S_{2})=\{[\tilde{U}_{n}]_{n\in\mathbb{N}}\}\cup \left(\bigsqcup_{k=1}^{p_2}\tilde{J}_{k}(\mathcal{E})\right).
\]

\vspace{4mm}
Let $[W_{n}]_{n\in \mathbb{N}}$ be an end of the fiber product $S_{1}\times_{(\beta_{1},\beta_{2})}S_{2}$, then we should prove that
\[
[W_{n}]_{n\in\mathbb{N}}\subset \{[\tilde{U}_{n}]_{n\in\mathbb{N}}\}\cup \left(\bigsqcup_{k=1}^{p_2}\tilde{J}_{k}(\mathcal{E})\right).
\]
For the compact $B_{n}$ described in equation (\ref{eq:special_compact}), we can suppose without of generality that $B_{n}\cap W_{n}=\emptyset$, for each $n\in\mathbb{N}$. It implies that the open connected subset $W_{n}\subset S_{1}\times_{(\beta_{1},\beta_{2})}S_{2}$ is contained in any connected component of
\[
S_{1}\times_{(\beta_{1},\beta_{2})}S_{2}\setminus B_{n}=\bigsqcup\limits_{j=1}^{i(n)} U_{j}^{n}.
\]
We study the following cases:

\vspace{1.5mm}
\noindent \emph{Case 1.} For each $n\in\mathbb{N}$, the open $W_{n}$ is contained in the connected component  $U_{1}^{n}=\tilde{U}_{n}$. It implies that the ends $[W_{n}]_{n\in\mathbb{N}}$ and $[\tilde{U}_{n}]_{n\in\mathbb{N}}$ are equivalent. In other words, $[W_{n}]_{n\in\mathbb{N}}=[\tilde{U}_{n}]_{n\in\mathbb{N}}$.

\vspace{1.5mm}
\noindent \emph{Case 2.} There is $l\in\mathbb{N}$ such that for the open $\tilde{U}_{l}=U^{l}_{1}$, it satisfies $W_{l}\cap U^{l}_{1}=\emptyset$, then the open $W_{n}$ is contained in the disjoint union $\bigsqcup\limits_{j=2}^{i(n)} U_{j}^{n}$. It implies that
\begin{equation}\label{eq:relation-1-ends}
[W_{n}]_{n\in\mathbb{N}}\subset \bigsqcup\limits_{j=2}^{i(n)} (U_{j}^{n})^{\ast}.
\end{equation}
Given that $B_{1}\subset B_{n}$ for each $n>1$, then $S_{1}\times_{(\beta_{1},\beta_{2})}S_{2}\setminus B_{1}\supset S_{1}\times_{(\beta_{1},\beta_{2})}S_{2}\setminus B_{n}$, it implies that
\[
\bigsqcup\limits_{j=2}^{i(n)} U_{j}^{n}\subset \bigsqcup\limits_{j=2}^{i(1)} U_{j}^{1}.
\]
By property of the open subsets of the ends space we have that
\begin{equation}\label{eq:relation-2-end}
\bigsqcup\limits_{j=2}^{i(n)} (U_{j}^{n})^{\ast}\subset \bigsqcup\limits_{j=2}^{i(1)} (U_{j}^{1})^{\ast}=\left(\bigsqcup_{k=1}^{p_2}\tilde{J}_{k}(\mathcal{E})\right).
\end{equation}
From the relations describe in (\ref{eq:relation-1-ends}) and (\ref{eq:relation-2-end}) we conclude that
\[
[W_{n}]_{n\in\mathbb{N}}\in\left(\bigsqcup_{k=1}^{p_2}\tilde{J}_{k}(\mathcal{E})\right).
\]
\end{proof}

An immediate result is the following Corollary, of which we present a proof with different ideas from those used for the proof of Theorem \ref{theorem:ends_space_of_fiber_product2}.

\begin{corollary}
Suppose that the locus of singular points of the fiber product $S_{1}\times_{(\beta_{1},\beta_{2})}S_{2}$ is
\[
{\rm Sing}(S_{1}\times_{(\beta_{1},\beta_{2})}S_{2})=\{(z_{1},z_{2})\in S_{1}\times_{(\beta_{1},\beta_{2})}S_{2}: \beta_{1}(z_{1})=\beta_{2}(z_2)\in A_{2}\}.
\]
If $S_{1}$ is topologically equivalent to the Loch Ness monster and the subset $A_{2}\subset\mathbb{C}$ is infinite, then the fiber product $S_{1}\times_{(\beta_{1},\beta_{2})}S_2$ has only one end.
\end{corollary}

\begin{proof}
We take a compact subset $K\subset S_{1}\times_{(\beta_{1},\beta_{2})}S_{2}$, then we shall prove that there is a compact subset $K^{'}\subset S_{1}\times_{(\beta_{1},\beta_{2})}S_{2}$ such that $K\subset K^{'}$ and  $S_{1}\times_{(\beta_{1},\beta_{2})}S_{2}\setminus K^{'}$ is connected (see Lemma \ref{lemma:spec}). 

We consider the continuous projection map $\beta:S_{1}\times_{(\beta_{1},\beta_{2})}S_{2}\to\mathbb{C}$ described in equation (\ref{eq:cont_proj_prod_fib}), then the image $\beta(K)$ is a compact subset of $\mathbb{C}$. As the complex plane $\mathbb{C}$ has only one ends, then there exists a compact subset $B\subset \mathbb{C}$ such that $\beta(K)\subset B$ and $\mathbb{C}\setminus B$ is connected. Given that the projection $\beta$ is a proper map (see Remark \ref{remarK:projection_map_beta_proper}), then the inverse image $K^{'}=\beta^{-1}(B)$ is a compact subset of $S_{1}\times_{(\beta_{1},\beta_{2})}S_{2}$ such that $K\subset K^{'}$. Finally, we should prove that $S_{1}\times_{(\beta_{1},\beta_{2})}S_{2}\setminus K^{'}$ is connected. We will fix a point $\textbf{w}_{0}$ in $S_{1}\times_{(\beta_{1},\beta_{2})}S_{2}\setminus K^{'}$ and for any $\textbf{z}\in S_{1}\times_{(\beta_{1},\beta_{2})}S_{2}\setminus K^{'}$, then we will define a path with end points $\textbf{z}$ and $\textbf{w}_{0}$, such that it does not intersect to $K^{'}$, this will imply that $S_{1}\times_{(\beta_{1},\beta_{2})}S_{2}\setminus K^{'}$ is connected. Given that $A_{2}$ is infinite, we can suppose that there is a positive integer $M\in\mathbb{N}$ such that $z_{m}\notin B$, for each $m\geq M$. We denote by $\textbf{w}_{0}$ the inverse image of $z_{M}$ under the projection $\beta$. If $\textbf{z}$ is a point in $S_{1}\times_{(\beta_{1},\beta_{2})}S_{2}\setminus K^{'}$, then there exists a path $\gamma$ in $\mathbb{C}\setminus B$ with end points $\beta(\textbf{z})$ and $z_{M}$. Using the projection $\beta$, we can lift $\gamma$ to a path $\tilde{\gamma}$ in $S_{1}\times_{(\beta_{1},\beta_{2})}S_{2}\setminus K^{'}$ with end points $\textbf{z}$ and $\textbf{w}_{0}$.

\end{proof}

As consequence of Corollary \ref{corollary:properties_irreducible_components2}, we hold a relation between the irreducible components associated to the fiber product and topological type of the surfaces $S_{1}$ and $S_{2}$.

\begin{corollary}
The irreducible component $R_{k}$ of the fiber product $S_{1}\times_{(\beta_{1},\beta_{2})}S_{2}$ is biholomorphically equivalent to the Riemann surface $S_{1}$, for each $k\in\{1,\ldots,p_{2}\}$.
\end{corollary}

%%%%%%%%%%%%%%%%%%%%%%%%%%%%%%%%%%%%%%%%%%%%%%%%%%%%%%%%%%%%%%%%%%%%%%%%%%%%%%
%%%%%%%%%%%%%%%%%%%%%%%%%%%%%%%%%%%%%%%%%%%%%%%%%%%%%%%%%%%%%%%%%%%%%%%%%%%%%%
%%%%%%%%%%%%%%%%%%%%%%%%%%%%%%%%%%%%%%%%%%%%%%%%%%%%%%%%%%%%%%%%%%%%%%%%%%%%%%
%%%%%%%%%%%%%%%%%%%%%%%%%%%%%%%%%%%%%%%%%%%%%%%%%%%%%%%%%%%%%%%%%%%%%%%%%%%%%%

\section{Fiber product over infinite superelliptic curves}\label{sec:fiber_product_infinite_superelliptic_curve}

 Let $(w_l)_{l\in\mathbb{N}}$ be a sequence of different complex numbers such that $\lim\limits_{l\to\infty}|w_l|=\infty$. Then by the Weierstrass theorem in \cite{Palka}*{p. 498} there exists a meromorphic function $f:{\mathbb C} \to \widehat{\mathbb C}$ (called \emph{Weierstrass map}) whose simple zeroes are given by the points $w_1,w_2,\ldots$. Moreover, $f$ is uniquely determined (up multiplication) by  a zero-free entire map (for example $e^{z}$). Such functions $f$ admit the representation 
\begin{equation*}\label{eq:T_Weierstrass}
f(z)=h(z)z^{m}\prod_{l=1, w_{l} \neq 0}^{\infty}\left(1-\frac{z}{w_l}\right)E_l(z),
\end{equation*}
where $h$ is a zero-free entire function ($m=0$ if $z_{l} \neq 0$ for every $l\in\mathbb{N}$; in the other case, $m=1$ for $z_{l}=0$), 
and $E_{l}(z)$ is a function of the form
\[
E_{l}(z)=\exp \left[\sum_{s=1}^{d(l)}\frac{1}{s}\left(\frac{z}{w_{l}}\right)^{s} \right],
\]
for a suitably large non-negative integer $d(l)$.

Now, if we consider the holomorphic function $F:{\mathbb C}^{2} \to {\mathbb C}$ given by $F(z_{1},z_{2})=z_{2}^{n}-f(z_{1})$, for $n\in\mathbb{N}$ such that $n \geq 2$, then we obtain the affine plane curve 
\begin{equation*}
S(f):=\left\{(z_{1},z_{2}) \in {\mathbb C}^{2}: z_{2}^{n}=f(z_{1})\right\}.
\end{equation*}

\begin{definition}[\cite{AGHQR}*{Subsection 6.3}]
The affine curve $S(f)$ is a Riemann surface, called {\bf infinite superelliptic curve}. Further, if $n=2$, the affine curve $S(f)$ is known as  \textbf{infinite hyperelliptic curve}.
\end{definition}

The following result describes the topology type of an infinite superelliptic curve. 

\begin{theorem}[\cite{AGHQR}*{Theorem 6.12}]\label{t:infinite_hyperelliptic_curve}
 The infinite superelliptic curve $S(f)$ is a connected Riemann surface homeomorphic to the Loch Ness monster.
\end{theorem}

The following result describes the conditions when two infinite hyperelliptic curves are biholomorphically equivalent.

\begin{theorem}[\cite{AGHQR}*{Theorem.6.13}]\label{equivalentehyperellipticcurves}
If $n=2$, that is, for the hyperelliptic case, the pairs $(S(f),G_f)$ and $(S(g),G_g)$ are biholomorphically equivalent if and only if there exists a holomorphic automorphism of complex plane $\mathbb{C}$ carrying the zeros of $f$ onto the zeros of $g$.
\end{theorem}

\subsection{Fiber product over infinite superelliptic curves.} We consider a sequence  of different complex numbers $(w_{l})_{l\in\mathbb{N}}$ such that $\lim \limits_{l\to \infty}\vert w_{l}\vert =\infty$. Let $f$ be denote the Weierstrass map with simple zeros at the points of the sequence $(w_{l})_{l\in\mathbb{N}}$. Now, we fix a subset $A$ of the complex plane $\mathbb{C}$ and we define a map $g$ as follows: 

\begin{itemize}
    \item[\textbf{(1)}] If $A$ is a finite set of even cardinality of points of the sequence $(w_{l})_{l\in\mathbb{N}}$, then   
\[
g(z_1)= \prod\limits_{a\in A}(z_1-a).
\]
\item[\textbf{(2)}] If $A$ is a subsequence of $(w_{l})_{l\in\mathbb{N}}$, which also diverges, then $g$ is a Weierstrass map with simple zeros at the points in $A$.
\end{itemize}

Associated to the holomorphic functions above, we have the following affine plane curves 
\begin{align}
    S(f) &= \left\{(z_{1},z_{2})\in\mathbb{C}^2:z_{2}^{p}=f(z_{1})\right\}, \label{eq:Cf_1}\\
    S(g)&=\left\{(z_1,z_3)\in\mathbb{C}^2: z_{3}^{q}=g(z_1)\right\}, \label{eq:Cg_1}
\end{align}
such that $p,q\geq 2$.
  
From the Theorem \ref{t:infinite_hyperelliptic_curve}, it follows that the infinite superelliptic curve  $S(f)$ described in equation (\ref{eq:Cf_1}) is a Riemann surface  topologically equivalent to the Loch Ness monster. Moreover, as consequence of the implicit function theorem  \cite{KraPa}, \cite{Miranda}*{p. 10, Theorem 2.1}, the affine plane curve $S(g)$ defined in the equation (\ref{eq:Cg_1}) admits a Riemann surface structure. It has the following topological type:

\begin{itemize}
\item[\textbf{(a)}] If the subset $A\subset \mathbb{C}$ is finite and $\vert A\vert=2k$, with $k\geq 1$, then  by Riemann-Hurwitz's Theorem \cite{Miranda}*{Theorem 4.16} the affine plane curve $S(g)$ has $q$ non-planar ends and genus equal to $k-1$. 

\item[\textbf{(b)}] If $A$ is a subsequence of $(w_{l})_{l\in\mathbb{N}}$, then by Theorem \ref{t:infinite_hyperelliptic_curve}, it follows that the affine plane curve  $S(g)$ is an infinite superelliptic curve topologically equivalent to the Loch Ness monster.
\end{itemize}

Now, we consider $\beta_{1}: S(f)\to \mathbb{C}$ and $\beta_{2}:S(g)\to\mathbb{C}$ the projection maps onto first coordinate and take the fiber product
\begin{equation}\label{eq:fiber_product}
S(f)\times_{(\beta_{1},\beta_{2})}S(g)=\left\{(z_{1},z_{2},z_{3})\in\mathbb{C}^3: z_{2}^{p}=f(z_1), \, z_{3}^{q}=g(z_1)\right\}.
\end{equation}
For convenience, in this section we will denote the above fiber product by  
\begin{equation}
\mathcal{S}(f,g):=S(f)\times_{(\beta_{1},\beta_{2})}S(g).
\end{equation}
Thus, we consider the projection map 
$\beta:\mathcal{S}(f,g)\to\mathbb{C}$,
defined as in equation (\ref{eq:cont_proj_prod_fib}), where
\begin{equation}\label{eq:projections_pi_1_2}
\beta=\beta_1\circ\pi_1=\beta_2\circ\pi_2,
\end{equation}
and the maps  $\pi_1:\mathcal{S}(f,g)\to S(f)$ and  $\pi_2:\mathcal{S}(f,g)\to S(g)$ are given by $\pi_1(z_1,z_2,z_{3})=(z_{1},z_{2})$ and $\pi_2(z_1,z_2,z_{3})=(z_{1},z_{3})$, respectively.

\begin{remark}\label{remark:proper_map}
The projection maps $\beta_{1}$ and $\beta_{2}$ are proper branched covering maps. Moreover, for each $z \in \mathbb{C}\setminus \{w_{l}:l\in\mathbb{N}\}$, the fiber $\beta^{-1}_{1}(z)$ consists of $p\geq 2$ points; and for each $z \in \mathbb{C}\setminus A$, the fiber $\beta^{-1}_{2}(z)$ consists of $q\geq 2$ points. As consequence, the map $\beta$ is proper, because $\beta^{-1}(K)$ is a closed subset of the compact $K\times \left(h^{-1} [ f(K)] \right)\times h^{-1}[g(K)]$, where $h$ is the complex  $q$-th power map $h(w)=w^q$.
\end{remark}

 On the fiber product $\mathcal{S}(f,g)$ the implicit function theorem does not apply everywhere, more precisely, on the points $(z_1, z_2, z_3)$ such that $z_1 \in A$. By Proposition \ref{pro:RSA} we conclude that the fiber product $\mathcal{S}(f,g)$ is a singular Riemann surface whose locus of singular points is the discrete set
 \begin{equation}\label{eq:singular_points}
     {\rm Sing}(\mathcal{S}(f,g))=\left\{(z_1,z_2,z_3)\in \mathcal{S}(f,g): z_1\in A\right\},
 \end{equation}
and the subspace \begin{equation}\label{eq:ob_M}
\mathcal{S}(f,g)\setminus {\rm Sing}(\mathcal{S}(f,g))
 \end{equation} 
 of the fiber product $\mathcal{S}(f,g)$ is a Riemann surface.  
 
 From the Remark \ref{remark:proper_map} and Theorem \ref{theorem:ends_space_of_fiber_product} it holds immediately the following results, which describe the connectedness and the ends space of the fiber product, and the ends space of the normal fiber product.
 
\begin{theorem}\label{theorem:topology-of-fiber-product-superelliptic-curve}
The singular Riemann surface $\mathcal{S}(f,g)$ (given as in (\ref{eq:fiber_product})) is connected. Moreover, the end space of  $\mathcal{S}(f,g)$ has:
\begin{itemize}
    \item[\textbf{(1)}] one end, if  $A$ is infinite,
    
    \item[\textbf{(2)}] $q$ ends, if $A$ is finite.
\end{itemize}
\end{theorem}

\begin{corollary}
For each $k\in\{1,\ldots,q\}$, the irreducible component $R_{k}$ of the fiber product $\mathcal{S}(f,g)$ is biholomorphically equivalent to the Loch Ness monster $S(f)$.
\end{corollary}

Now, we give necessary and sufficient conditions for isomorphisms of arbitrary singular Riemann surfaces coming from hyperelliptic curves. 

We consider an arbitrary subset $B$ of $\{w_{l}:l\in\mathbb{N}\}$, which can be finite or infinite satisfying the same conditions of the set $A$, such subset also defines a Weierstrass function or a polynomial map $h$, depending of the cardinality of $B$, whose simple zeros are the complex numbers in $B$. Hence, as in equation (\ref{eq:Cg_1}) we define the affine plane curve 
\begin{equation}
S(h)=\left\{(z_1,z_3)\in\mathbb{C}^2:z_3^2=h(z_1)\right\}.
\end{equation}
Now, we consider the superellyptic curve $S(f)$ define in equation (\ref{eq:Cf_1}) and the projection maps onto the first coordinate $\beta_{1}: S(f)\to \mathbb{C}$ and $\beta_{3}:S(h)\to\mathbb{C}$, then we hold other singular Riemann surface as the fiber product
\begin{equation}\label{eq:fiber_product_second}
\mathcal{S}(f,h)=\left\{(z_{1},z_{2},z_{3})\in\mathbb{C}^3: z_{2}^{2}=f(z_1), \, z_{3}^{2}=h(z_1)\right\}.
\end{equation}

With all above it holds the next result:

\begin{theorem}\label{theorem:two-isomorphic-regular-riemann-surfaces}
The singular Riemann surfaces $\mathcal{S}(f,g)$ and $\mathcal{S}(f,h)$ are isomorphic if and only if  the set $A$ and $B$ have the same cardinality and  there exist a biholomorphism of the complex plane $t:\mathbb{C}\to \mathbb{C}$ such that
it leaves invariant the points of the sequence $(w_{l})_{l\in\mathbb{N}}$ and $t(A)=B$.
\end{theorem}

\begin{proof}
 We shall prove the \emph{necessary} condition. Suppose that the set  $A$ and $B$ have the same cardinality and there is a biholomorphism   $t:\mathbb{C}\to \mathbb{C}$ such that it leaves invariant the points of the sequence $(w_{n})_{n\in\mathbb{N}}$ and $t(A)=B$. Theorem \ref{equivalentehyperellipticcurves} implies that there exists a biholomorphism $T:S(f)\to S(f)$ satisfying the following properties:
 \begin{enumerate}
 \item[$\ast$] leaves invariant the subset  $\{(w_{l},\textbf{0}):l\in\mathbb{N}\}\subset S(f)$,
 \item[$\ast$]  sends the set $\{(a,\textbf{0}):a\in A\}$ onto the set $\{(b,\textbf{0}):b\in B\}$,
 \item[$\ast$] the following diagram
 \begin{equation}
\begin{tikzcd}
S(f)\arrow{d}[swap]{\beta_{1}} \arrow{r}{T}          &S(f)\arrow{d}{\beta_{1}}\\
\mathbb{C} \arrow{r}[swap]{t}          &\mathbb{C}
\end{tikzcd}
\end{equation}
 is commutative, \emph{i.e.}, $\beta_{1}\circ T =t\circ \beta_{1}$. 
 \end{enumerate}
 
We consider the projection maps  $\pi_{1}:\mathcal{S}(f,g)\to S(f)$ and $\overline{\pi}_{1}:\mathcal{S}(f,h)\to S(f)$, which send the triplet $(z_{1},z_{2},z_{3})$ on the pair $(z_1,z_{2})$. We will define the function $\tilde{T}:\mathcal{S}(f,g)\to \mathcal{S}(f,h)$ as a lifting of the map $T$, such that the following diagram
\begin{equation}
\begin{tikzcd}
{\mathcal{S}(f,g)} \arrow{d}[swap]{\pi_1} \arrow{r}{\tilde{T}} & \mathcal{S}(f,h) \arrow{d}{\overline{\pi}_1}\\
S(f) \arrow{r}[swap]{T}          &S(f)
\end{tikzcd}
\end{equation}
is commutative, \emph{i.e.}, $T\circ \pi_{1}=\overline{\pi}_{1}\circ \tilde{T}$. We fix ${\textbf z}\in \mathcal{S}(f,g)$ and chose $\textbf{w}=\tilde{T}(\textbf{z})$ such that $T\circ \pi_1(\textbf{z})=\overline{\pi}_{1}\circ \tilde{T}(\textbf{z})$. Now, we consider a point $\textbf{s}\in \mathcal{S}(f,g)$ and we shall define the image of $\textbf{s}$ under $\tilde{T}$. Then we take a path $\gamma$ in $S(f)$ having ends points $\pi_1(\textbf{z})$ and $\pi_{1}(\textbf{s})$. Hence, for the path $T\circ \gamma$ there is a lifting path $\tilde{\gamma}$ in $\mathcal{S}(f,h)$ having ends points $\tilde{T}(\textbf{z})$ and $\tilde{T}(\textbf{s})$. By analytic continuation the map $\tilde{T}:\mathcal{S}(f,g)\to \mathcal{S}(f,h)$ is homeomorphis such that $\tilde{T}({\rm Sin}(\mathcal{S}(f,g)))={\rm Sing}(\mathcal{S}(f,h))$ and the restriction $\tilde{T}:\mathcal{S}(f,g)\setminus {\rm Sing}(\mathcal{S}(f,g))\to \mathcal{S}(f,h)\setminus {\rm Sing}(\mathcal{S}(f,h))$ is a biholomorphims. Thus, we conclude that $\tilde{T}$ is a isomorphism.

Now, we shall prove the \emph{sufficient} condition. By hypothesis, there exists a homeomorphism  $\tilde{T}: \mathcal{S}(f,g)\to\mathcal{S}(f,h)$ such that
\begin{itemize}
    \item[\textbf{(1)}] $\tilde{T}({\rm Sin}(\mathcal{S}(f,g)))={\rm Sing}(\mathcal{S}(f,h))$. 
    \item[\textbf{(2)}] The restriction map $\tilde{T}:\mathcal{S}(f,g)\setminus {\rm Sing}(\mathcal{S}(f,g))\to \mathcal{S}(f,h)\setminus {\rm Sing}(\mathcal{S}(f,h))$ is a biholomorphic.
\end{itemize}
Recall that the set of singular points of $\mathcal{S}(f,h)$ and $\mathcal{S}(f,h)$ are the sets ${\rm Sing}(\mathcal{S}(f,g))=\{(z_1,z_2,z_3)\in\mathcal{S}(f,g): z_1\in A\}$ and ${\rm Sing}(\mathcal{S}(f,h))=\{(z_1,z_2,z_3)\in\mathcal{S}(f,h):z_{1}\in B\}$, respectively.

From the proof of Theorem \ref{theorem:ends_space_of_fiber_product}, it follows that each connected component of 
$$\mathcal{S}(f,g)\setminus {\rm Sing}(\mathcal{S}(f,g))$$
is biholomorphic equivalent to the punted Loch Ness monster $$S(f)\setminus \{(z_1,z_2)\in S(f): z_{1}\in A\}.$$ Analogously, each connected component of $$\mathcal{S}(f,h)\setminus {\rm Sing}(\mathcal{S}(f,h))$$
is biholomorphic equivalent to the punted Loch Ness monster
$$S(f)\setminus \{(z_1,z_2)\in S(f): z_1\in B\}.$$
Hence, the restriction of $\tilde{T}$ to each connected component of $\mathcal{S}(f,g)\setminus {\rm Sing}(\mathcal{S}(f,g))$ defines a biholomorphism $T$ from $S(f)\setminus \{(z_1,z_2)\in S(f):z_{1}\in A\}$ to $S(f)\setminus\{(z_{1},z_{2})\in S(f): z_1\in B\}$, as the sets of points $\{(z_1,z_2)\in S(f):z_{1}\in A\}$ and $\{(z_{1},z_{2})\in S(f): z_1\in B\}$ of $S(f)$ are discrete, then the map $T$ can be extended to a biholomirphic map $H:S(f)\to S(f)$. From Theorem \ref{equivalentehyperellipticcurves} it follows that there exists a biholomophic $t$ of the complex plane $\mathbb{C}$ such that the following diagram
\begin{equation*}
\begin{tikzcd}
{S(f)} \arrow{d}[swap]{\beta_1} \arrow{r}{H} & S(f) \arrow{d}{\beta_1}\\
\mathbb{C} \arrow{r}[swap]{t}          &\mathbb{C}
\end{tikzcd}
\end{equation*}
is commutative \emph{i.e.}, $t\circ \beta_{1}=\beta_{1}\circ H$. As the isomorphism $\tilde{T}$ sends the singular points of $\mathcal{S}(f,g)$ onto the singular points of $\mathcal{S}(f,h)$ and $H$ is a holomorphic extension, then it must be happen $t(A)=B$. 
\end{proof}

\subsection{Double covering of infinite hyperelliptic curves.}  
We consider the fiber product of surfaces (see equation (\ref{eq:fiber_product})): 
 \begin{equation*}
\mathcal{S}(f,g)=\left\{(z_{1},z_{2},z_{3})\in\mathbb{C}^3: z_{2}^{2}=f(z_1), \, z_{3}^{2}=g(z_1)\right\},
\end{equation*}
 which admits two automorphisms $\alpha_{1}$ and $\alpha_{2}$ of order two, which are  given by
\begin{equation}
    \alpha_{1}(z_{1},z_{2},z_{3})=(z_{1},z_{2},-z_{3}),\quad \alpha_{2}(z_{1},z_{2},z_{3})=(z_{1},-z_{2},z_{3}).
\end{equation}
The group $\langle \alpha_1, \alpha_2 \rangle$ is a subgroup of ${\rm Aut}(\mathcal{S}(f,g))$ isomorphic to $\mathbb{Z}_2\oplus\mathbb{Z}_2$. So, the cyclic subgroup $\langle \alpha_{i}\rangle$ acts on the singular Riemann surface $\mathcal{S}(f,g)$, for each $i\in\{1,2\}$. 

On the other hand, let us remember that: if $S_1$ is a surface (singular or not) and  $G$ is a finite group, a {\em Galois cover} of $S_1$ with group $G$, shortly a $G$-{\em cover} of $S_1$, is a finite morphism $\pi:S_2\to S_1$, where $S_2$ is a surface (singular or not) with an effective action by $G$, such that $\pi$ is $G$-invariant and induces an isomorphism $S_2/G\cong S_1$.

With all above, the following result ensures that the quotient surface $\mathcal{S}(f,g)/\langle \alpha_{i} \rangle$ is a Riemann surface and describes its topological type.

\vspace{3mm}

\begin{lemma}\label{lemma:quotient}
The quotient surfaces $\mathcal{S}(f,g)/\langle \alpha_{1} \rangle$ and $\mathcal{S}(f,g)/\langle \alpha_{2} \rangle$ are Riemann surfaces biholomophic to $S(f)$ and $S(g)$, respectively.
\end{lemma}

\begin{proof}
We consider the $\langle \alpha_1\rangle$-cover  $p_{1}:\mathcal{S}(f,g)\to \mathcal{S}(f,g)/\langle \alpha_{1} \rangle$. If we consider the projection function $\pi_{1}: \mathcal{S}(f,g) \to S(f)$ as in equation (\ref{eq:projections_pi_1_2}), then we have the following diagram
\begin{equation}
\begin{tikzcd}
\mathcal{S}(f,g) \arrow{d}[swap]{\pi_1} \arrow{r}{p_1} &\mathcal{S}(f,g)/\langle \alpha_{1} \rangle \arrow{ld}{\pi_{1}\circ p_{1}^{-1}} \\
S(f)          &
\end{tikzcd}
\end{equation}
where $\pi_{1}\circ p_{1}^{-1}$ is constant in the  fibres, and by the Transgression Theorem described in \cite{Dugu}*{p.123}, it follows that $\pi_{1}\circ p_{1}^{-1}$ is a biholomorphism. So, the surface $\mathcal{S}(f,g)/\langle \alpha_{1} \rangle$ is a Riemann surface biholomorphic to the Loch Ness monster $S(f)$.

Analogously, considering the $\langle \alpha_2\rangle$-cover $p_{2}:\mathcal{S}(f,g)\to \mathcal{S}(f,g)/\langle \alpha_{2} \rangle$, and the projection function $\pi_{2}: \mathcal{S}(f,g) \to S(g)$ as in equation (\ref{eq:projections_pi_1_2}), we observe that the map  $\pi_{2}\circ p_{2}^{-1}: \mathcal{S}(f,g)/\langle \alpha_{2} \rangle\to S(g)$ is constant in the  fibres, thus by the Transgression Theorem, it follows that $\pi_{2}\circ p_{2}^{-1}$ is a biholomorphism. So, the surface $\mathcal{S}(f,g)/\langle \alpha_{2} \rangle$ is a Riemann surface biholomorphic to the Riemann surface $S(g)$.
\end{proof}

\begin{proposition}\label{teo-cubrientes-dobles}
Let $S$ be a singular Riemann surfaces that is a $\mathbb{Z}_2\oplus\mathbb{Z}_2$-covering of the complex plane $\mathbb{C}$, with $\mathbb{Z}_2\oplus\mathbb{Z}_2\cong\langle \alpha_1,\alpha_2\rangle$. If $\pi_{\alpha_i}:S\to S_i=S/\langle \alpha_i\rangle$, $\beta_{1}:S_1\to S/\mathbb{Z}_2\oplus\mathbb{Z}_2$ and $\beta_{2}:S_2\to S/\mathbb{Z}_2\oplus\mathbb{Z}_2$ are the obvious projection maps,  then the singular Riemann surface $S$ is isomorphic to the fibre product $\mathcal{S}(f,g)$ defined by the following commutative diagram %$\pi_{\alpha_i}$
\begin{equation}
\begin{tikzcd}
S \arrow{d}[swap]{\pi_{\alpha_1}} \arrow{r}{\pi_{\alpha_2}} &S_2\arrow{d}{\beta_2}\\
S_1 \arrow{r}[swap]{\beta_1}          &S/\mathbb{Z}_2\oplus\mathbb{Z}_2%\mathbb{C}
\end{tikzcd}
\end{equation}
\end{proposition}

\begin{proof} By results above we have that $S_1$ is biholomorphic to $S(f) =\left\{(z_{1},z_{2})\in\mathbb{C}^2:z_{2}^{2}=f(z_{1})\right\}$, $S_2$ is biholomorphic to $S(g) =\left\{(z_1,z_3)\in\mathbb{C}^2: z_{3}^{2}=g(z_1)\right\}$ and recall that $S/\mathbb{Z}_2\oplus\mathbb{Z}_2$ is biholomorphic to complex plane $\mathbb{C}$.

%{\color{red}
%$\pi_{1}\to \pi_{\alpha_{1}}$

%$\pi_{1}\to \pi_{\alpha_{1}}$

%$q_{1}\to \pi_{1}$

%$q_{2} \to \pi_{2}$

%$\tilde{S}\to \mathcal{S}(f,g)$

%}

By the universal property of fibre products, for any commutative diagram of morphisms between non-compact Riemann surfaces as follows
%\begin{table}[h!]\label{diagram_2}
\begin{equation}
\begin{tikzcd}
{\mathcal{S}(f,g)} \arrow{d}[swap]{\pi_1} \arrow{r}{\pi_2} &S(g)\arrow{d}{\beta_2}\\
S(f) \arrow{r}[swap]{\beta_1}          &\mathbb{C}
\end{tikzcd}
\end{equation}
% \caption{}
%\end{table}
we shall prove that there is a unique morphism $\phi:\mathcal{S}(f,g)\to S$ making the following commutative diagram 
%\begin{table}[h!]\label{cac}
\begin{equation}\label{d4}
\begin{tikzcd}
\mathcal{S}(f,g)
\arrow[bend left]{drr}{\pi_2}
\arrow[bend right,swap]{ddr}{\pi_1}
\arrow[dashed]{dr}[description]{\phi} & & \\
& S \arrow{r}{\pi_{\alpha_2}} \arrow{d}[swap]{\pi_{\alpha_1}}
& S(g) \arrow{d}{\beta_2} \\
& S(f) \arrow[swap]{r}{\beta_1}
& \mathbb{C}
\end{tikzcd}
\end{equation}
%\caption{}
%\end{table}

We remark that, if such morphism $\phi$ exists, then the above diagrams are commutative, \emph{i.e.}, $(\pi_{\alpha_i}\circ \phi)(x)=\pi_i(x)$, for all $x\in \mathcal{S}(f,g)$. Thus, it must have that the image $\phi(x)\in \pi_{\alpha_i}^{-1}(\pi_i(x))$, $i=1, 2$. 

%Now, if such morphism $\phi$ is going to exist, we must have $(\pi_i\circ \phi)(x)=q_i(x)$, hence $\phi(x)\in \pi_i^{-1}(q_i(x))$, $i=1, 2$. 
Therefore, the result would follow if we could prove that, for all points $x\in \mathcal{S}(f,g)$, the intersection set $\pi_{\alpha_1}^{-1}(\pi_1(x))\cap \pi_{\alpha_2}^{-1}(\pi_2(x))$ contains exactly one point $\tilde{x}$ of $S$, in that case the morphism $\phi$ would be defined as  $\phi(x) = \tilde{x}$. Let us prove that this intersection contains a single point.\\

\textbf{Step 1.} The restriction of $\pi_{\alpha_1}$ to $\pi_{\alpha_2}^{-1}(\pi_2(x))$ is injective.\\

Let us assume that this restriction is not injective. Then, let us take the points $P$ and $Q$ in $\pi^{-1}_{\alpha_2}(\pi_{2}(x))$ such that $\pi_{\alpha_1}(P)=\pi_{\alpha_1}(Q)$. From where $\pi_{\alpha_i}(P)=\pi_{\alpha_i}(Q)$ con $i=1,2$. If we suppose $P\neq Q$, as $\alpha_{1}$ and $\alpha_{2}$ are automorphisms (differents) of $S$ of order two, we have $\alpha_{i}(P)=Q$, so $\alpha^{-1}_{1}\circ \alpha_{2}(P)=P$, this implies that $\alpha^{-1}_{1}\circ \alpha_{2}={\rm Id}$, therefore, $\alpha_{1}=\alpha_{2}$,
which is a contradiction since $\alpha_{1}\neq \alpha_{2}$. From where it follows that the restriction of $\pi_{\alpha_1}$ to $\pi_{\alpha_2}^{-1}(\pi_2(x))$ is injective.\\

\textbf{Step 2.}   The set $\pi_{\alpha_1}^{-1}(\pi_1(x))\cap \pi_{\alpha_2}^{-1}(\pi_2(x))$ contains, at most, one point.\\

Let us note that for every element $P\in \pi_{\alpha_1}^{-1}(\pi_1(x))\cap \pi_{\alpha_2}^{-1}(\pi_2(x))\subset \pi_{\alpha_2}^{-1}(\pi_2(x))\subset S$ the image $\pi_{\alpha_1}(P)=\pi_{1}(x)$ (see Diagram 4). As the function $\pi_{\alpha_1}$ restricted to the set $\pi_{\alpha_2}^{-1}(\pi_2(x))$ it is injective, it must happen that there is only one element at the intersection $\pi_{\alpha_1}^{-1}(\pi_1(x))\cap \pi_{\alpha_2}^{-1}(\pi_2(x))$. \\

\textbf{Step 3.}  It holds that $\pi_{\alpha_1}(\pi_{\alpha_2}^{-1}(\pi_2(x)))\subset \beta_1^{-1}(\beta_1(\pi_1(x)))\subset S(f)$.\\

Taking $P\in \pi_{\alpha_2}^{-1}(\pi_{2}(x))\subset S$ we must prove that $\pi_{\alpha_1}(P)\in \beta_{1}^{-1}(\beta_{1}( \pi_{1}(x)))$. From Diagram \ref{d4} we have the relationships $\beta_{1}(\pi_{\alpha_1}(P))=\beta_{2}(\pi_{\alpha_2}(P))=\beta_{2}(\pi_{2}(x))=\beta_{1}(\pi_{1}(x))$, as $\beta_{1}(\pi_{\alpha_1}(P))=\beta_1(\pi_1(x))$ then the point $\pi_{\alpha_1}(P)$ belongs to the fiber $\beta_{1}^{-1}(\beta_1(\pi_{1}(x)))$. \\

\textbf{Step 4.} For all points $x\in X$, we have $\pi_{\alpha_1}(\pi_{\alpha_2}^{-1}(\pi_2(x)))=\beta_1^{-1}(\beta_1(\pi_1(x)))$.\\

As the functions $\pi_{\alpha_i}$ y $\beta_{i}$ are of degree two, that is, each fiber has at most two points, then the sets $\beta_{1}^{-1}(\beta_{1}(\pi_{1}(x)))$ and $\pi_{\alpha_2}^{-1}(\pi_{2}(x))$ have at most two points. As $\pi_{\alpha_1}$ is injective (see step 1), then, we have that the set $\pi_{\alpha_1}(\pi_{\alpha_2}^{-1}(\pi_{2}(x)))$ it has at most two points, that is $\beta_1^{-1}(\beta_1(\pi_1(x)))\subset \pi_{\alpha_1}(\pi_{\alpha_2}^{-1}(\pi_2(x)))$. Using the contention described in step 3, we conclude that $\pi_{\alpha_1}(\pi_{\alpha_2}^{-1}(\pi_2(x)))=\beta_1^{-1}(\beta_1(\pi_1(x)))$.\\

\textbf{Step 5.} For all points $x\in X$, the set $\pi_{\alpha_1}^{-1}(\pi_1(x))\cap \pi_{\alpha_2}^{-1}(\pi_2(x))$ is non-empty.\\

As in the previous step, each point $\pi_{1}(x)\in \beta_{1}^{-1}(\beta_{1}(\pi_{1}(x)))$ can be written as $\pi_{1}(x)=\pi_{\alpha_1}(P)$, for some $P\in \pi_{\alpha_2}^{-1}(\pi_{2}(x))$. Then $P\in\pi_{\alpha_1}^{-1}(\pi_{1}(x))$, in particular $P\in\pi_{\alpha_1}^{-1}(\pi_{1}(x))\cap \pi_{\alpha_2}^{-1}(\pi_{2}(x))$. 

With all above it follows that $\phi$ is an isomorphism.
\end{proof}

Since every branching value of $\beta_{2}$ is also a branching value of $\beta_{1}$ with same branching order two at all ramification points, we may conclude that $\pi_{1}$ and $\pi_{2}$ are an unramified double covering of $S(f)$ and $S(g)$, respectively, (see \cite{Serre}*{p. 116} for details). Then, from  result above, we hold the following observation:

\begin{remark}
The singular Riemann surface $\mathcal{S}(f,g)$ is an  unramified double cover of the Loch Ness monster $S(f)$ and the surface $S(g)$.
\end{remark}\label{t:smooth__double_covering}

%%%%%%%%%%%%%%%%%%%%%%%%%%%%%%%%%%%%%%%%%%%%%%%%%%
%%%%%%%%%%%%%%%%%%%%%%%%%%%%%%%%%%%%%%%%%%%%%%%%
%%%%%%%%%%%%%%%Acknowledgements
%%%%%%%%%%%%%%%%%%%%%%%%%%%%%%%%%%%%%%%%%%%%%%%%%%
%%%%%%%%%%%%%%%%%%%%%%%%%%%%%%%%%%%%%%%%%%%%%%%%%%%
\section*{Acknowledgements}

Camilo Ram\'irez Maluendas was partially supported by UNIVERSIDAD NACIONAL DE COLOMBIA, SEDE MANIZALES. He has dedicated this work to his beautiful family: Marbella and Emilio, in appreciation of their love and support.

%%%%%%%%%%%%%%%%%%%%%%%%%%%%%%%%%%%%%%%%%%%%%%%%%%%
%%%%%%%%%%%%%%%%%%%%%%%%%%%%%%%%%%%%%%%%%%%%%%%%%%
%%%%%%%%BIBLIOGRAPHY
%%%%%%%%%%%%%%%%%%%%%%%%%%%%%%%%%%%%%%%%%%%%%%%%%
%%%%%%%%%%%%%%%%%%%%%%%%%%%%%%%%%%%%%%%%%%%%%%%%%

\end{document}